\documentclass[11pt]{amsart}
\usepackage{a4wide}
\usepackage{amssymb}
\usepackage{amsmath}
\usepackage{amsfonts}
\usepackage{amsthm}
\usepackage{graphicx}
\usepackage{epsfig}
\graphicspath{{./Figures/}}
\usepackage{longtable}
\usepackage{paralist}
\usepackage[usenames,dvipsnames]{colortbl}
\usepackage[usenames,dvipsnames]{color}
\usepackage{hyperref}
\usepackage{caption}

\newtheorem{thm}{Theorem}
\newtheorem{lem}[thm]{Lemma}
\newtheorem{cor}[thm]{Corollary}
\newtheorem{prop}[thm]{Proposition}

\theoremstyle{definition}
\newtheorem{defn}[thm]{Definition}

\theoremstyle{remark}
\newtheorem*{rmk}{Remark}

\newcommand{\E}{\mathbb{E}}

\newcommand{\DEF}{{:=}}
\newcommand{\FED}{{=:}}

\newcommand{\Sp}{\mathbb{S}}

\newcommand{\IL}{\mathbb{L}}
\newcommand{\IP}{\mathbb{P}}

\DeclareMathOperator{\dd}{\mathrm{d}}

\DeclareMathOperator{\discr}{discr}

\newcommand{\MARKED}[2]{{\textcolor{#1}{#2}}}

\newcommand{\bsp}{\boldsymbol{p}}
\newcommand{\bsPhi}{\boldsymbol{\Phi}}
\newcommand{\bsw}{\boldsymbol{w}}
\newcommand{\bsx}{\boldsymbol{x}}
\newcommand{\bsy}{\boldsymbol{y}}
\newcommand{\bsz}{\boldsymbol{z}}

\allowdisplaybreaks[1]

\title[Point sets on the sphere $\mathbb{S}^2$ with small spherical cap discrepancy]{Point sets on the sphere $\mathbb{S}^2$ with small spherical cap discrepancy} 
\author[Ch. Aistleitner, J. S. Brauchart, and J. Dick]{Ch. Aistleitner\textdaggerdbl, J. S. Brauchart\textasteriskcentered, and J. Dick\textdagger} 
\thanks{\noindent \textdaggerdbl The research of this author was supported by the Austrian Research Foundation (FWF), Project
S9603-N23. \\  \textasteriskcentered The research of this author was supported by an Australian Research Council Discovery Project. \\
\textdagger The research of this author was supported by an Australian Research Council Queen Elizabeth 2 Fellowship.}

\date{\today}

\DeclareCaptionLabelFormat{andtable}{#1??\ #2 and \tablename??\ \thetable}

\begin{document}

\address{Ch. Aistleitner, Graz University of Technology, Institute of
Mathematics A, Steyrergasse 30, 8010 Graz, Austria \\ 
J. S. Brauchart and J. Dick:
School of Mathematics and Statistics,
University of New South Wales,
Sydney, NSW, 2052,
Australia }
\email{aistleitner@math.tugraz.at}
\email{j.brauchart@unsw.edu.au}
\email{josef.dick@unsw.edu.au}

\begin{abstract}
In this paper we study the geometric discrepancy of explicit
constructions of uniformly distributed points on the two-dimensional
unit sphere. We show that the spherical cap discrepancy of random point sets, of spherical
digital nets and of spherical Fibonacci lattices
converges with order $N^{-1/2}$. Such point sets are therefore
useful for numerical integration and other computational
simulations. The proof uses an area-preserving Lambert map. A detailed analysis of the level curves and sets of the pre-images of spherical caps under this map is given.
\end{abstract}

\keywords{Discrepancy, Isotropic Discrepancy, Lambert Map, Level Curve, Level Set, Numerical Integration, quasi Monte Carlo, Spherical Cap Discrepancy} 
\subjclass[2000]{Primary 65D30, 65D32; Secondary }

\maketitle


\section{Introduction}

Let $\mathbb{S}^2 = \{ \bsz \in \mathbb{R}^3: \| \bsz \| = 1 \}$ be the unit sphere in the Euclidean space $\mathbb{R}^3$ provided with the norm $\| \cdot \|$ induced by the usual inner product $\bsx \cdot \bsy$. On this sphere we consider the Lebesgue surface area measure $\sigma$ normalised to a probability measure ($\int_{\mathbb{S}^2} \dd \sigma = 1$).

This paper is concerned with uniformly distributed sequences of points on $\mathbb{S}^2$. Informally speaking, a sequence of points is called uniformly distributed if every reasonably defined (clopen) $A \subseteq \mathbb{S}^2$ gets a fair share of points as their number $N$ grows.
Given a triangular scheme $\{ \bsz_{1,N}, \dots, \bsz_{N,N} \}$, $N \geq 1$, of points on $\mathbb{S}^2$ in such a case one has
\begin{subequations}
\begin{equation} \label{eq:uniform.a}
\lim_{N \to \infty} \frac{\mathrm{card}(\{ j : \bsz_{j,n} \in A \})}{N} = \sigma( A ),
\end{equation}
(where $\mathrm{card}$ denotes the cardinality of the set) or, equivalently (defined in terms of numerical integration)
\begin{equation} \label{eq:uniform.b}
\lim_{N \to \infty} \frac{1}{N} \sum_{j=1}^N f( \bsz_{j,N} ) = \int_{\mathbb{S}^2} f \dd \sigma \qquad \text{for every $f$ continuous on $\mathbb{S}^2$.}
\end{equation}
\end{subequations}
The degree of uniformity is quantified by the so called {\em spherical cap discrepancy}.

A spherical cap $C = C(\bsw,t)$ centred at $\bsw \in
\mathbb{S}^2$ with height $t \in [-1,1]$ is given by the set
\begin{equation*}
C(\bsw, t) = \left\{\bsy \in \mathbb{S}^2 : \bsw \cdot \bsy > t \right\}.
\end{equation*}
(We assume that spherical caps are open subsets of $\mathbb{S}^2$.) 
The boundary of $C(\bsw,t)$ then is 
\begin{equation*}
\partial C(\bsw,t) = \left\{\bsy \in \mathbb{S}^2: \bsw \cdot \bsy = t \right\}.
\end{equation*}

Let $Z_N = \{\bsz_0, \ldots, \bsz_{N-1}\} \subseteq \mathbb{S}^2$ be an $N$-point set on the sphere $\mathbb{S}^2$.
The local discrepancy with respect to a spherical cap $C$ measures
the difference between the proportion of points in $C$ (the
empirical measure of $C$) and the normalised surface area of $C$. The spherical cap discrepancy is then the supremum of the local discrepancy over all spherical caps, as stated in the following definition.

\begin{defn}
The {\em spherical cap discrepancy} of an $N$-point set $Z_N = \{\bsz_0,\ldots, \bsz_{N-1}\} \subseteq \mathbb{S}^2$ is
\begin{equation*}
D(Z_N) = \sup_{\bsw \in \mathbb{S}^2} \sup_{-1 \le t \le 1} \left| \frac{1}{N}\sum_{n=0}^{N-1} 1_{C(\bsw,t)}(\bsz_n) - \sigma(C(\bsz,t))\right|.
\end{equation*}
\end{defn}

If the point set $Z_N$ is well distributed, then this discrepancy is small. In fact, a sequence of $N$-point systems $(Z_N)_{N \geq 1}$ satisfying 
\begin{equation} \label{eq:D.to.0}
\lim_{N \to \infty} D(Z_N) = 0,
\end{equation}
is called {\em asymptotically uniformly distributed}. Using, for example, the classical {\em Erd{\"o}s-Tur{\'a}n type inequality} (cf. Grabner~\cite{Gr1991}, also cf. Li and Vaaler~\cite{LiVa1999}) or {\em LeVeque type inequalities} (Narcowich, Sun, Ward, and Wu~\cite{NaSuWa2010}) and the fact that the set of polynomials is dense in the set of continuous functions, one can show that \eqref{eq:D.to.0} is equivalent with \eqref{eq:uniform.b}. 

It is known from \cite{Be1984} 
that there are constants $c, C > 0$, independent of $N$, such that a {\em low-discrepancy} scheme $\{ Z_N^* \}_{N \geq 2}$ satisfies
\begin{equation} \label{eq:opt.D.C}
c \, N^{-3/4} \le D(Z_N^*) \le C \, N^{-3/4} \sqrt{\log N}.
\end{equation}
The lower bound holds for all $N$-point sets $Z_N$ on $\mathbb{S}^2$ and there always exists an $N$-point set $Z_N \subseteq \mathbb{S}^2$ such that the upper bound holds. The proof of the upper bound is probabilistic in nature and is thus non-constructive. \MARKED{Black}{To our best knowledge explicit constructions of low-discrepancy schemes are not known.} (In this paper we restrict ourselves to the sphere $\mathbb{S}^2$, though some of the results are known for spheres of dimension $d \ge 2$.)

An explicit construction of points $Z_N$ with small spherical cap
discrepancy has been given in \cite{LuPhSa1986, LuPhSa1987}. For instance, in \cite{LuPhSa1986} it was shown that
\begin{equation} \label{eq:LuPhSa.bound}
D(Z_N) \le C (\log N)^{2/3} N^{-1/3}.
\end{equation}
The numerical experiments in \cite{LuPhSa1986} indicate a convergence rate of $\mathcal{O}(N^{-1/2})$.

In this paper we give explicit constructions of point sets $Z_N$ for
which we have
\begin{equation*}
D(Z_N) \le 44 \sqrt{2} N^{-1/2}.
\end{equation*}
Our numerical results indicate a convergence rate of $\mathcal{O}((\log N)^c N^{-3/4})$ for some $1/2 \le c \le 1$, see Tables~\ref{table2}, \ref{table1} below.

The {\em spherical cap $\mathbb{L}_2$-discrepancy}
\begin{equation*} 
D_{\IL_{2}}(Z_{N}) \DEF \left\{ \int_{-1}^{1} \int_{\Sp^{2}} \left| \frac{1}{N}\sum_{n=0}^{N-1} 1_{C(\bsw,t)}(\bsz_n) - \sigma(C(\bsz,t))\right|^{2} \dd \sigma(\bsw) \, \dd{t} \right\}^{1/2},
\end{equation*}
which averages the local discrepancy for a spherical cap over all caps, provides a lower bound for the spherical cap discrepancy. It is closely related to the {\em sum of distances} and its continuous counterpart the {\em distance integral} by means of {\em Stolarsky's invariance principle~\cite{St1973}} for the Euclidean distance and the $2$-sphere,
\begin{equation*} 
\frac{1}{N^2} \sum_{j = 1}^N \sum_{k = 1}^N \left| \bsz_j - \bsz_k \right| + 4 \left[ D_{\IL_{2}}^{C}(\bsz_1, \dots, \bsz_N) \right]^2 = \int_{\mathbb{S}^2} \int_{\mathbb{S}^2} \left| \bsz - \bsw \right| \dd \sigma(\bsz) \, \dd \sigma(\bsw) \FED V_{-1}( \mathbb{S}^2 ) = \frac{4}{3}.
\end{equation*}
This gives a simple way of computing the spherical cap $\mathbb{L}_2$-discrepancy of point sets on $\mathbb{S}^2$. In \cite{BrWo20xx} it is shown that the spherical cap $\mathbb{L}_2$-discrepancy of $Z_N$ can be interpreted as the worst-case error of an equal weight numerical integration rule with node set $Z_N$ for functions in the unit ball of a certain Sobolev space over $\mathbb{S}^2$. It is shown in \cite{BrSaSlWo20xx} that on average (that is for randomly chosen points independently identically uniformly distributed over the sphere), the expected squared worst-case error is of the form $(4/3) N^{-1}$. Thus the expected value of the squared spherical cap discrepancy satisfies
\begin{equation} \label{exp}
8 \mathbb{E} \left[ D(Z_N) \right]^2 \geq 4 \mathbb{E} \left[ D_{\IL_{2}}(Z_{N}) \right]^2 = \frac{4}{3} \, N^{-1}.
\end{equation}
We study the expected value and the {\em typical} asymptotic order of the spherical cap discrepancy of random point sets in detail in Section \ref{sec:rand}. Among other results, we show that the there is also a constant $C > 0$ such that $\mathbb{E}\left[ D(Z_N)\right] \le C N^{-1/2}$. \\

Point configurations maximising the sum of distances, by Stolarsky's invariance principle, have low spherical cap $\IL_2$-discrepancy. It is known from \cite{Be1984} that low spherical cap $\IL_2$-discrepancy point sets satisfy relations similar to \eqref{eq:opt.D.C} except for the logarithmic term introduced by the probabilistic approach. The upper bound for the spherical cap discrepancy of maximum sum of distances points $\hat{Z}_N^*$ obtained in \cite{NaSuWa2010} is much weaker but still better than \eqref{eq:LuPhSa.bound}: For some positive constant $c > 0$, not depending on $N$,
\begin{equation*}
D( Z_N^* ) \leq c N^{-3/8}.
\end{equation*}

For point configurations $Z_N^*$ emulating electrons restricted to move on $\mathbb{S}^2$ in the most stable equilibrium, that is minimising their {\em Coulomb potential energy} essentially given by
\begin{equation*}
\mathop{\sum_{j=1}^N \sum_{k=1}^N}_{j \neq k} \frac{1}{\left| \bsz_j - \bsz_k \right|},
\end{equation*}
one can show the bound 
\begin{equation*}
D( Z_N^* ) \leq C N^{-1/2} \log N.
\end{equation*}
The estimate $D( Z_N^* ) = \mathcal{O}(N^{-1/2})$ was conjectured by Korevaar~\cite{Ko1996} and later proved (up to the logarithmic factor) by G{\"o}tz~\cite{Go2000}. When allowing so-called {\em $K$-regular} test sets \footnote{Roughly speaking, $\sigma$-measurable subsets $B$ of $\mathbb{S}^2$ whose $\delta$-neighbourhoods $\partial_\delta B$ relative to $\mathbb{S}^2$ satisfy $\sigma( \partial_\delta B ) \leq K \delta$. For example, spherical caps are $K_1$-regular for some fixed $K_1$.} introduced by Sj{\"o}gren~\cite{Sj1972}, the estimate above is sharp in the following sense: The upper bound holds for any $K$-regular test set, whereas there are some numbers $K_0$ and $c$ such that to any $N$ points $\bsz_1, \dots, \bsz_N \in \mathbb{S}^2$ there is a $K_0$-regular test set $B$ with \cite[Corollary~2]{Go2000}
\begin{equation*}
c \, K_0 \, N^{-1/2} \leq \left| \frac{1}{N} \sum_{n=1}^{N} 1_{B}(\bsz_n) - \sigma(B) \right|.
\end{equation*}
(The lower bound also applies to the explicit constructions given in this paper.) Bounds for the spherical cap discrepancy of so-called minimal Riesz energy configurations (for the concept of Riesz energy see, for example Saff and Kuijlaars~\cite{KuSa1997} and Hardin and Saff~\cite{HaSa2004}) can be found in \cite{Br2008} (for the logarithmic energy), Damelin and Grabner~\cite{DaGr2003,DaGr2004} (the first hyper-singular case), and \cite{NaSuWa2010} (sums of generalised distances). Wagner~\cite{Wa1992b} estimates the spherical cap discrepancy in terms of the Riesz energy. It should be mentioned that there are very few known explicit constructions of point configurations with optimal Riesz energy. In general, one has to rely on numerical optimisation to generate such point sets. The underlying  (constrained) optimisation problem is highly non-linear. Moreover, numerical results indicate that the number of local minima increases exponentially with the number of points. (For the computational complexity see, for example, Bendito et al.~\cite{BeCa2009}.)

{\em Spherical $n$-designs} introduced by Delsarte, Goethals and Seidel in the landmark paper \cite{DeGoSei1977} are node sets for equal weight numerical integration rules such that all spherical polynomials of degree $\leq n$ are integrated exactly. Grabner and Tichy~\cite{GrTi1993} give the following upper bound of the spherical cap discrepancy of a spherical $n$-design with $N(n)$ points
\begin{equation} \label{eq:discr.bound.sph.design}
D( Z_{N(n)}^* ) \leq C n^{-1}
\end{equation}
which immediately follows from the aforementioned {\em Erd{\"o}s-Tur{\'a}n type inequality}. (See also Andrievskii, Blatt, and G{\"o}tz~\cite{AnBlGo1999} for a similar form for $K$-regular test sets.)

A spherical $n$-design is the solution of a system of polynomial equations (one for every spherical harmonic of the real orthonormal basis of the space of spherical polynomials of degree $\leq n$). Hence, a natural lower bound for the number of points of a spherical $n$-design is given by the dimension of the involved polynomial space; that is, one needs at least $\geq n^2 / 4$ points. The famous conjecture that $C \, n^2$ points (for some universal $C > 0$) are sufficient for a spherical $n$-design seems to have been settled by Bondarenko, Radchenko and Viazovska~\cite{BoRaVi2011arXiv}. The proposed proof is non-constructive. Hardin and Sloane~\cite{HaSl1996} propose a construction of so-called putative spherical $n$-designs with $(1/2) \, n^2 + o(n^2)$ points. The variational characterisation of spherical designs introduced in \cite{SlWo2009} (also cf. \cite{GrTi1993}) leads to a minimisation problem for a certain energy functional (changing with $n$) whose minimiser is a spherical $n$-design if and only if the functional becomes zero. Numerical results also suggest a coefficient $1/2$. When allowing more points, $N(n) = (n + 1)^2$, interval-based methods yield, in principle, the existence of a spherical $n$-design near so-called extremal (maximum determinant points, cf. \cite{SlWo2004}). Due to the computational cost this approach was carried out only for $n \leq 20$. Very recently, Chen, Frommer, and Lang~\cite{ChFrLa2011} devised a computational algorithm based on interval arithmetic that, upon successful completion, verifies the existence of a spherical $n$-design with $(n + 1)^2$ points and provides narrow interval enclosures which are known to contain these nodes with mathematical certainty. The spherical cap discrepancy of all such obtained spherical $n$-design with $\mathcal{O}(n^2)$ points can then be bounded by $C^\prime \, N^{-1/2}$ by \eqref{eq:discr.bound.sph.design}. For the sake of completeness it should be mentioned that the tensor product rules used by Korevaar and Meyers~\cite{KoMe1993} to prove the existence of spherical $n$-designs of $N(n) = \mathcal{O}(n^3)$ points give rise to $N(n)$-point configurations whose spherical cap discrepancy can be bounded by $C^{\prime\prime} \, [N(n)]^{-1/3}$ by \eqref{eq:discr.bound.sph.design}.

%

From \cite{BrWo20xx} it follows that the spherical cap discrepancy
of a point set $Z_N = \{\bsz_0,\ldots, \bsz_{N-1}\} \subseteq
\mathbb{S}^2$ yields an upper bound on the integration error in
certain Sobolev spaces of functions defined on $\mathbb{S}^2$ using
a quadrature rule $Q_N(f) = \frac{1}{N} \sum_{n=0}^{N-1} f(\bsz_n)$.
Thus, our results here provide an explicit mean of finding
quadrature points for numerical integration of functions defined on
$\mathbb{S}^2$. Our result here improves the bound on the
integration error in \cite{BD2011} by a factor of $\sqrt{\log N}$. 

The construction of the points on $\mathbb{S}^2$ is obtained by
mapping low-discrepancy points on $[0,1]^2$ to $\mathbb{S}^2$ using
an equal area transformation $\bsPhi:[0,1]^2 \to \mathbb{S}^2$. The
same approach has previously been used in \cite{BD2011} and \cite{HN2004}, in both cases in the context of numerical integration.  The low-discrepancy points in $[0,1]^2$ are obtained
from digital nets and Fibonacci lattices, see \cite{DiPi2010,
Ni1992}. These point sets are well-distributed with respect to
rectangles anchored at the origin $(0,0)$. However, the set
$$\bsPhi^{-1}(C(\bsw,t)) = \{\bsx \in [0,1]^2: \bsPhi(\bsx) \in
C(\bsw,t)\}$$ is, in general, not a rectangle. In fact, it is not
even a convex set, although the boundary
\begin{equation*}
\bsPhi^{-1}(\partial C(\bsw,t)) = \{\bsx \in [0,1]^2: \bsPhi(\bsx) \in \partial C(\bsw,t)\}
\end{equation*}
is a continuous curve.

Hence, in order to prove bounds on the spherical
cap discrepancy of digital nets and Fibonacci lattices lifted to the
sphere using $\bsPhi$ (we call those point sets spherical digital nets
and spherical Fibonacci lattices), we need to prove bounds on a
general notion of discrepancy in $[0,1]^2$. To this end we study discrepancy in $[0,1]^2$ with respect to  convex sets, the corresponding discrepancy is known as {\it isotropic discrepancy} \cite{BeCh1987}. We show that digital nets and Fibonacci lattices have isotropic discrepancy of order $\mathcal{O}(N^{-1/2})$.  Using these result and some properties of
the function $\bsPhi$, we can show that spherical digital nets and
spherical Fibonacci lattices have spherical cap discrepancy at most
$C N^{-1/2}$ for an explicitly given constant $C$, see Corollary~\ref{cor1} and \ref{cor2}. Note that the best possible rate of convergence of the isotropic discrepancy is $N^{-2/3} (\log N)^c$ for some $0\le c \le 4$, see \cite{Be1988} and \cite[p.107]{BeCh1987}. Hence the approach via the isotropic discrepancy cannot give the optimal rate of convergence for the spherical cap discrepancy.

In the following we define the equal area Lambert map
$\bsPhi$ and show some of its properties.

\section{The equal-area Lambert transform and some properties}

The points on the sphere are obtained by using the {\em Lambert cylindrical equal-area projection} 
\begin{equation} \label{eq:Lambert.map}
\bsPhi(\alpha,\tau) = \left(2 \sqrt{\tau- \tau^2} \, \cos (2\pi \alpha), 2 \sqrt{\tau- \tau^2} \, \sin (2\pi \alpha), 1-2 \tau \right), \qquad \alpha,\tau \in [0,1].
\end{equation}

The area-preserving Lambert map can be illustrated in the following way. The unit square $[0,1]^2$ is linearly stretched to the rectangle $[0,2 \pi] \times [-1,1]$, rolled into a cylinder of radius $1$ and height $2$ and fitted around the unit sphere such that the polar axis is the main $z$-axis. This way a point $(\alpha, \tau)$ in $[0,1]^2$ is mapped to a point on the cylinder which is radially projected along a ray orthogonal to the polar axis onto the sphere giving the point $\bsPhi(\alpha,\tau)$.

Axis-parallel rectangles in the unit square are mapped to spherical ``rectangles'' of equal area, see Figure~\ref{fig:rectangles}.

\begin{figure}[ht]
\begin{center} 
\includegraphics[scale=.08155]{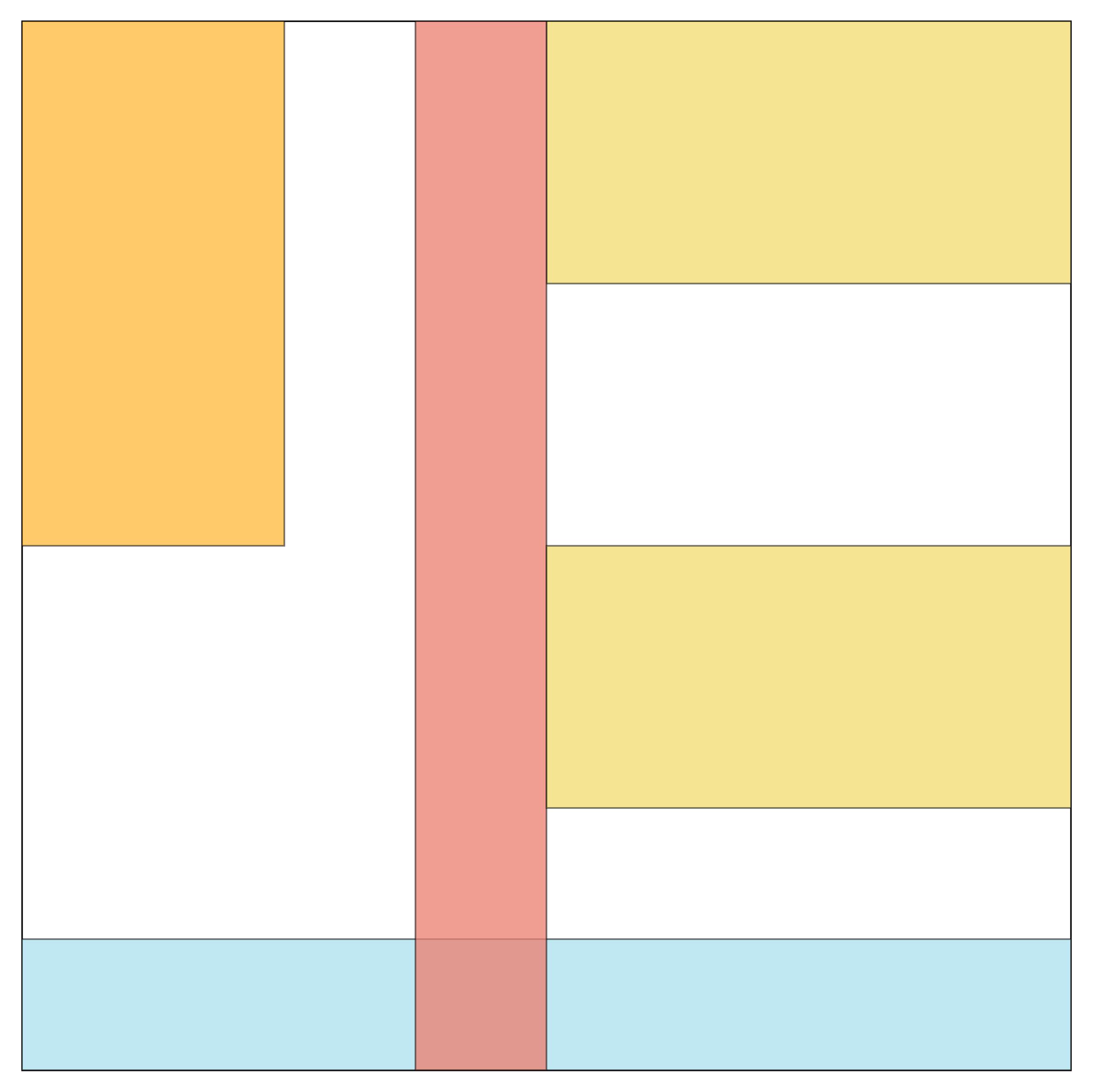}
\includegraphics[scale=.075]{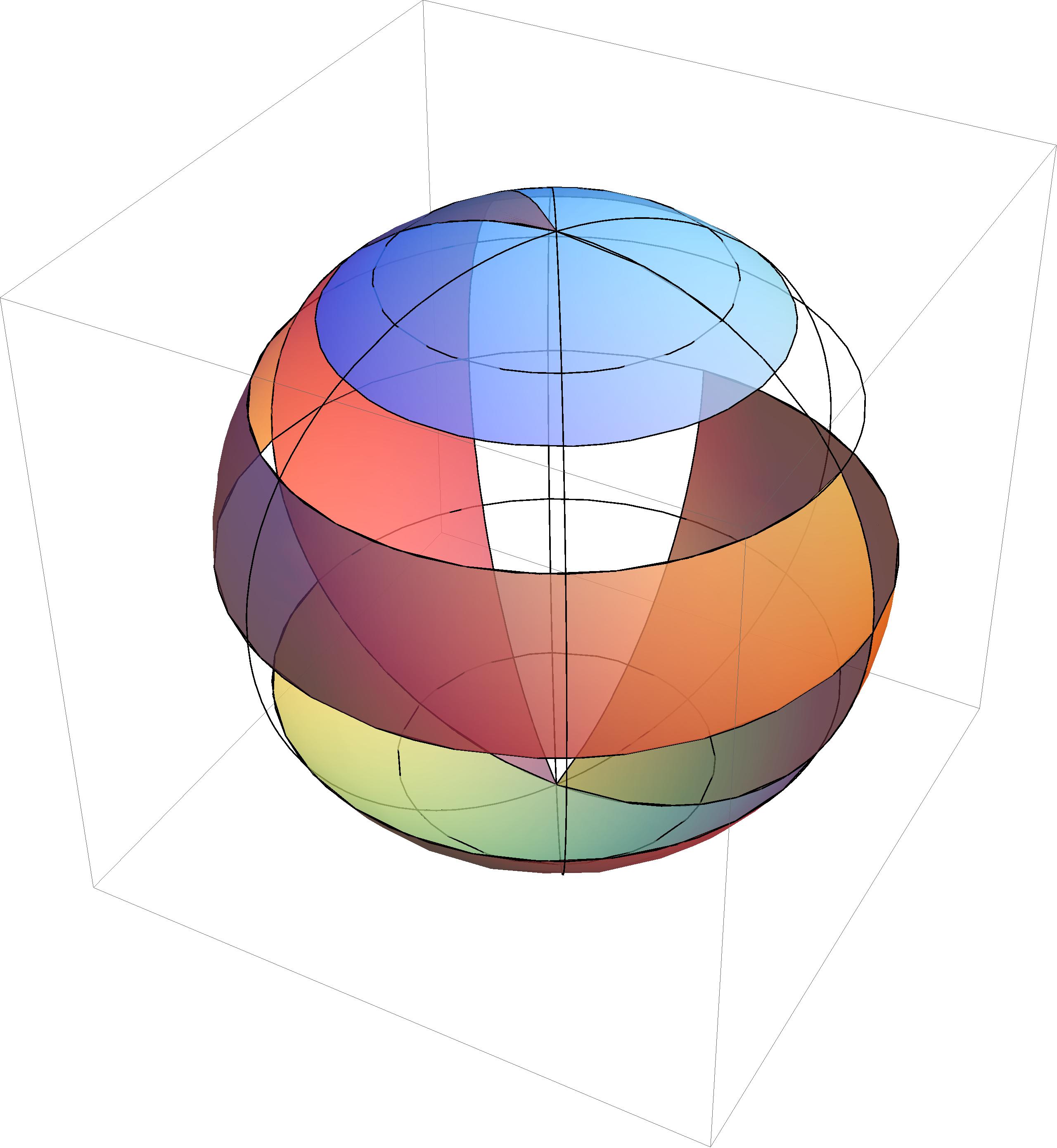}
\caption{\label{fig:rectangles} Axis-parallel rectangles in the square and their images under $\bsPhi$ on $\mathbb{S}^2$.}
\end{center}
\end{figure}

The pre-images of a spherical cap centred at $\bsw$ with height $t$ under the Lambert map is the set 
\begin{equation*}
B(\bsw,t) = \bsPhi^{-1}(C(\bsw,t)) = \left\{ (\alpha, \tau) \in [0,1) \times [0,1] : \bsPhi(\alpha,\tau) \in C(\bsw,t) \right\}
\end{equation*}
and the pre-image of the boundary of this spherical cap is
\begin{equation*}
\partial B(\bsw,t) = \bsPhi^{-1}(\partial C(\bsw,t)) = \left\{ (\alpha, \tau) \in [0,1) \times [0,1] : \bsPhi(\alpha,\tau) \in \partial C(\bsw,t) \right\}.
\end{equation*}

The sets $B(\bsw,t)$ are not convex, in general. Thus, we consider a more general class of sets which we call pseudo-convex. A definition is given in the following.

\begin{defn} \label{def:pseudo-convex}
Let $A$ be an open subset of $[0,1]^2$ such that there exists a collection of $p$ convex subsets $A_1$, \dots, $A_p$ of $[0,1]^2$ with the following properties: 
\begin{inparaenum}[(\itshape a\upshape)]
\item $A_j \cap A_k$ is empty for $j \neq k$;
\item $A \subseteq A_1 \cup \cdots \cup A_p$; 
\item either $A_j$ is a convex part of $A$ ($A_j \subseteq A$) or the complement of $A$ with respect to $A_j$, $A_j^\prime = A_j \setminus A$, is convex.
\end{inparaenum}
Then $A$ is called a {\em pseudo-convex} set and $A_1, \dots, A_p$ is an admissible convex covering for $A$ with $p$ parts (with $q$ convex parts of $A$).
\end{defn}

\begin{lem} \label{lem2}
For every $\bsw \in \mathbb{S}^2$ and all $-1 \leq t \leq 1$ the pre-image $B(\bsw,t)$ of the spherical cap $C(\bsw,t)$ centred at $\bsw$ with height $t$ under the Lambert map is pseudo-convex with an admissible convex covering with at most $7$ parts. More precisely, taking into account the number of convex parts of the pre-image, among the convex coverings with $p$ parts and $q$ of which are convex the worst case has $p = 7$ and $q = 3$ which implies the constant $2p - q = 11$.
\end{lem}

The proof of Lemma~\ref{lem2} in Section~\ref{sec:proofs} gives details how to construct admissible coverings.


\section{Isotropic- and spherical cap discrepancy}

We introduce the isotropic discrepancy of a point set and a sequence as follows. Let $\lambda$ be the Lebesgue area measure in the unit square.

\begin{defn}
The {\em isotropic discrepancy} $J_N$ of an $N$-point set $P_N= \{ \bsx_0, \ldots, \bsx_{N-1} \}$ in $[0,1]^2$ is defined as
\begin{equation*}
J_N(P_N) = \sup_{A \in \mathcal{A}} \left| \frac{1}{N} \sum_{n=0}^{N-1} 1_{A}(\bsx_n) - \lambda(A) \right|,
\end{equation*}
where $\mathcal{A}$ is the family of all convex subsets of $[0,1]^{2}$.

For an infinite sequence $\bsx_0,\bsx_1,\ldots \in [0,1]^2$ the isotropic discrepancy is defined as the isotropic discrepancy of the initial $N$ points of the sequence.
\end{defn}

\begin{lem} \label{lem:pseudo.convex}
Let $A$ be a pseudo-convex subset of $[0,1]^2$ with an admissible convex covering of $p$ parts with $q$ convex parts of $A$. Then for any $N$-point set $P_N = \{\bsx_0,\ldots, \bsx_{N-1}\} \subseteq [0,1]^2$
\begin{equation*}
\left| \frac{1}{N} \sum_{n=0}^{N-1} 1_A( \bsx_n ) - \lambda(A) \right| \leq \left( 2 p - q \right) J_N( P_N ).
\end{equation*}
\end{lem}

\begin{proof}
Let $A_1, \dots, A_p$ be an admissible convex covering of $A$ with $p$ parts. Without loss of generality let $A_1, \dots, A_q$ be the convex parts of $A$ and $A_{q+1}, \dots, A_p$ those for which $A_j^\prime = A_j \setminus A$ ($q + 1 \leq j \leq p$) is convex. Clearly
\begin{equation*}
A = \bigcup_{j = 1}^q A_j \cup \bigcup_{j = q + 1}^p \left( A_j \setminus A_j^\prime \right).
\end{equation*}
Thus
\begin{align*}
&\frac{1}{N} \sum_{n=0}^{N-1} 1_A( \bsx_n ) - \lambda(A) 
\\
&\phantom{equals}= \sum_{j=1}^q \Big[ \frac{1}{N} \sum_{n=0}^{N-1} 1_{A_j}( \bsx_n ) - \lambda(A_j) \Big] + \sum_{j = q + 1}^p \Big[ \frac{1}{N} \sum_{n=0}^{N-1} 1_{A_j \setminus A_j^\prime}( \bsx_n ) - \lambda(A_j \setminus A_j^\prime) \Big] \\
&\phantom{equals}= \sum_{j=1}^q \Big[ \frac{1}{N} \sum_{n=0}^{N-1} 1_{A_j}( \bsx_n ) - \lambda(A_j) \Big] + \sum_{j = q + 1}^p \Big[ \frac{1}{N} \sum_{n=0}^{N-1} 1_{A_j}( \bsx_n ) - \lambda(A_j) \Big] \\
&\phantom{equals\pm}- \sum_{j = q + 1}^p \Big[ \frac{1}{N} \sum_{n=0}^{N-1} 1_{A_j^\prime}( \bsx_n ) - \lambda(A_j^\prime) \Big].
\end{align*}
In the last line all sets are convex and we can use the isotropic discrepancy in the estimation
\begin{equation*}
\left| \frac{1}{N} \sum_{n=0}^{N-1} 1_A( \bsx_n ) - \lambda(A) \right| \leq \left[ q + \left( p - q \right) + \left( p - q \right) \right] J_N(\bsx_0, \dots, \bsx_{N-1}).
\end{equation*}
\end{proof}

\begin{thm} \label{theoremDJ}
Let $P_N = \{\bsx_0,\ldots, \bsx_{N-1}\} \subseteq [0,1]^2$ and let $Z_N = \{\bsPhi(\bsx_0),\ldots, \bsPhi(\bsx_{N-1})\} \subseteq \mathbb{S}^2$. Then
\begin{equation*}
D(Z_N) \leq 11 J_N(P_N).
\end{equation*}
\end{thm}

\begin{proof}
Let $\bsw \in \mathbb{S}^2$ and $-1 \leq t \leq 1$. A point
$\bsPhi(\bsx_n) \in C(\bsw,t)$ if and only if $\bsx_n \in B(\bsw,t)$.
Thus
\begin{equation*}
\sum_{n=0}^{N-1} 1_{C(\bsw,t)}(\bsPhi(\bsx_n)) = \sum_{n=0}^{N-1} 1_{B(\bsw,t)}(\bsx_n).
\end{equation*}
Further, since the transformation $\bsPhi$ preserves areas, we have
\begin{equation*}
\sigma(C(\bsw,t)) = \lambda(B(\bsw,t)).
\end{equation*}
Hence
\begin{equation*}
\left|\frac{1}{N} \sum_{n=0}^{N-1} 1_{C(\bsw,t)}(\bsPhi(\bsx_n)) - \sigma(C(\bsw,t)) \right|  = \left| \frac{1}{N} \sum_{n=0}^{N-1} 1_{B(\bsw,t)}(\bsx_n) - \lambda(B(\bsw,t)) \right|.
\end{equation*}
The pre-images are pseudo-convex in the sense of Definition~\ref{def:pseudo-convex} by Lemma~\ref{lem2}. Applying Lemma~\ref{lem:pseudo.convex} with the constant $2p - q = 11$ from Lemma~\ref{def:pseudo-convex} we arrive at the result.
\end{proof}

We have now reduced the problem of proving bounds on the spherical cap discrepancy to prove bounds on the isotropic discrepancy of points in the square $[0,1]^2$. We will study this problem in Section \ref{sec:iso}.

\section{Spherical cap discrepancy of random points sets} 
\label{sec:rand}

Let $(M, \mathcal{M})$ be a measurable space, and let $P$ be a probability on it. Let further $X_n$, $n \geq 0$, denote a sequence of independent, identically distributed (i.i.d.) random variables on a probability space $(\Omega, \mathcal{A}, \IP)$ with values in $M$, and let $\mathcal{C} \subseteq \mathcal{M}$ denote a class of subsets of $M$. To avoid measurability problems we will assume throughout the rest of this section that the class $\mathcal{C}$ is countable. Let $A \subseteq M$ be an arbitrary set. Then $\mathcal{C}$ is said to {\em shatter} $A$ if to every possible subset $B$ of $A$ there exists a set $C \in \mathcal{C}$ such that
\begin{equation*}
C \cap A = B.
\end{equation*}
For $k \geq 1$ the $k$-th {\em shattering coefficient} $S_{\mathcal{C}}(k)$ of $\mathcal{C}$ is defined as
\begin{equation*}
S_{\mathcal{C}}(k) \DEF \max_{x_1, \dots, x_k \in M} \mathrm{card} \{ \{x_1, \dots, x_k\} \cap C : C \in \mathcal{C} \}.
\end{equation*}
The {\em Vapnik-\v{C}ervonenkis dimension} (VC-dimension) of $\mathcal{C}$ is defined as
\begin{equation*}
v (\mathcal{C}) \DEF \min_k \{ k : S_{\mathcal{C}} < 2^k \}.
\end{equation*}
(Here we use the convention that the minimum of the empty set is $\infty$.) A class $\mathcal{C}$ with finite VC-dimension is called a {\em Vapnik-\v{C}ervonenkis class} (VC class). The theory of VC classes is of extraordinary importance in the theory of empirical processes indexed by classes of functions. For example, a class $\mathcal{C}$ is {\em uniformly Glivenko-Cantelli} if and only if it is a VC class, see \cite{vc}. We will use the following theorem, which is a combination of results of Talagrand \cite[Theorem 6.6]{tala} and Haussler \cite[Corollary 1]{hauss}, and has already been used by Heinrich {\em et al.} \cite{hein} in the context of probabilistic discrepancy theory. 
\begin{thm}[{see \cite[Theorem 2]{hein}}] \label{thmtala}
There exists a positive number $K$ such that for each VC class $\mathcal{C}$ and each probability $P$ and sequence $X_n$, $n \geq 0$, as above the following holds: For all $s \geq K \sqrt{v (\mathcal{C})}$ we have
\begin{equation*}
\IP \left\{ \sup_{C \in \mathcal{C}} \left| \frac{1}{N} \sum_{n=0}^{N-1} 1_C (X_n) - P(C)\right| \geq \frac{s}{\sqrt{N}} \right\} \leq \frac{1}{s} \left(\frac{K s^2}{v(\mathcal{C})} \right)^{v(\mathcal{C})} e^{-2s^2}.
\end{equation*}
\end{thm}
In our setting we will have $M=\mathbb{S}^2$, $\mathcal{M}$ will denote the sigma-field generated by the class of spherical caps, $P$ will stand for the normalised Lebesgue surface area measure $\sigma$, and $\mathcal{C}$ will denote the class of all spherical caps for which the centre $\boldsymbol{w}$ is a vector of rational numbers and the height $t$ is also a rational number (this restriction is necessary to assure that the class $\mathcal{C}$ is countable; of course the spherical cap discrepancy with respect to this class is the same as the discrepancy with respect to the class of {\em all} spherical caps). In the sequel we assume that the i.i.d. random variables $X_n$, $n \geq 0$, are uniformly distributed on $\mathbb{S}^2$. We will write $Z_N=Z_N(\omega)$ for the (random) point set $\{ X_0, \dots, X_{N-1} \} = \{ X_0 (\omega), \dots, X_{N-1}(\omega) \}$.

The following proposition asserts that the class $\mathcal{C}$ is a VC class (the proof of this and the subsequent results of this section can be found in Section \ref{sec:proofs}).
\begin{prop} \label{propo}
The class $\mathcal{C}$ has VC dimension 5.
\end{prop}
Using Theorem \ref{thmtala} and Proposition \ref{propo} we can prove the following results:

\begin{thm} \label{tha1}
There exist constants $C_1, C_2$ such that for $N \geq 1$
\begin{equation*}
C_1 \sqrt{N} \leq \E \left[ D(Z_N)\right] \leq C_2 \sqrt{N}.
\end{equation*}
\end{thm}
\begin{rmk}
The existence of such a constant $C_1$ for the lower bound follows directly from (\ref{exp}); we can choose $C_1 = 6^{-1/2}$.
\end{rmk}

\begin{thm} \label{tha2}
For any $\varepsilon > 0$ there exist positive constants $C_3(\varepsilon), C_4(\varepsilon)$ such that for sufficiently large $N$
\begin{equation*}
\IP \left\{ C_3 \leq \sqrt{N} D(Z_N) \leq C_4 \right\} \geq 1 - \varepsilon.
\end{equation*}
\end{thm}
Theorem \ref{tha2} shows that the {\em typical} discrepancy of a random set of $N$ points is of order $N^{-1/2}$. However, actually much more is true, since by classical results any VC class $\hat{\mathcal{C}}$ on a measurable space $(\hat{M},\hat{\mathcal{M}})$ is a so-called {\em Donsker class}, which essentially means that for every probability measure $\hat{P}$ and every sequence $V_n$, $n \geq 0$, of i.i.d. random variables having law $\hat{P}$ the {\em empirical process indexed by sets} 
\begin{equation*}
\alpha_N (C) = \sqrt{N} \left| \frac{1}{N} \sum_{n=0}^{N-1} 1_{C}(V_n) - P(C) \right|, \qquad C \in \hat{\mathcal{C}}
\end{equation*}
converges weakly to a centered, bounded Gaussian process $B(C)$, which has covariance structure 
\begin{equation*}
\mathbb{E} B(C_1) B(C_2) = P (C_1 \cap C_2) - P(C_1) P(C_2), \qquad C_1, C_2 \in \hat{\mathcal{C}}. 
\end{equation*}
This weak convergence could, for example, be used to prove the existence of a limit distribution of $\sqrt{N} D(Z_N)$ as $N \to \infty$; however, to keep this presentation short and self-contained we will not pursue this method any further, and refer the interested reader to \cite{alex2, dud1, dud2, tala2} and the references therein.

\begin{rmk}
The upper bounds in Theorem \ref{tha1} and \ref{tha2} follow from Theorem \ref{thmtala}. However, since no concrete value for the constant $K$ in Theorem \ref{thmtala} is known, the value of the constants $C_2$ and $C_4$ in Theorem \ref{tha1} and Theorem \ref{tha2}, respectively, is also unknown. It is possible that the decomposition technique from \cite{aist} can be used to achieve a version of Theorem \ref{tha1} and \ref{tha2} with explicitly known constants in the upper bound.
\end{rmk}

Finally, the following theorem describes the asymptotic order of a {\em typical} infinite sequence of random points.
\begin{thm} \label{tha3}
We have 
\begin{equation*}
D (Z_N) = \mathcal{O} \left( \frac{\sqrt{\log \log N}}{\sqrt{N}} \right) \qquad \textrm{as $N \to \infty$}, \qquad \textrm{almost surely}.
\end{equation*}
\end{thm}
Theorem \ref{tha3} is a so-called {\em bounded law of the iterated logarithm}, and follows easily from Theorem \ref{theoremDJ} and Philipp's law of the iterated logarithm (LIL) for the isotropic discrepancy of random point sets in the plane. More precisely, Philipp \cite{philipp} proved that for a sequence of i.i.d. uniformly distributed random variables $Y_n$, $n \geq 0$, on the unit square (writing $P_N$ for the (random) point set $\{ Y_0, \dots, Y_{N-1} \}$) the law of the iterated logarithm
\begin{equation*}
\limsup_{N \to \infty} \frac{N J_N(P_N)}{\sqrt{2 N \log \log N}} = \frac{1}{2} \qquad \textup{a.s.}
\end{equation*}
holds. Together with Theorem \ref{theoremDJ} this implies for $Z_N = \{\bsPhi(Y_0),\ldots, \bsPhi(Y_{N-1})\} \subseteq \mathbb{S}^2$
\begin{equation*}
\limsup_{N \to \infty} \frac{N D(Z_N)}{\sqrt{2 N \log \log N}} \leq \frac{11}{2} \qquad \textup{a.s.},
\end{equation*}
which proves Theorem \ref{tha3} (it is necessary to observe that the image of a sequence of i.i.d. uniformly distributed random variables on the unit square under the area-preserving Lambert map is a sequence of i.i.d. uniformly distributed random variables on the sphere). It is easy to see that Theorem \ref{tha3} is optimal, except for the value of the implied constant. More precisely, let $C^*$ denote a fixed spherical cap with area $2 \pi$ (which means that $C^*$ is a hemisphere, and has normalised surface area measure $\sigma(C^*)=1/2$). Then clearly the random variables 
\begin{equation*}
1_{C^*}(\bsPhi(Y_{n})) - \sigma(C^*), \qquad n \geq 0,
\end{equation*}
have expected value 0 and variance 1/4. Thus by the classical law of the iterated logarithm for sequences of i.i.d. random variables
\begin{equation*}
\limsup_{N \to \infty} \frac{\left| \frac{1}{N}\sum_{n=0}^{N-1} 1_{C^*}(\bsPhi(Y_{n})) - 1/2 \right|}{\sqrt{2 N \log \log N}} = \frac{1}{2} \qquad \textup{a.s.},
\end{equation*}
and since 
\begin{equation*}
D(Z_N) \geq \left| \frac{1}{N}\sum_{n=0}^{N-1} 1_{C^*}(\bsPhi(Y_{n})) - \sigma(C^*) \right|, 
\end{equation*}
we finally arrive at
\begin{equation*}
\limsup_{N \to \infty} \frac{N D(Z_N)}{\sqrt{2 N \log \log N}} \geq \frac{1}{2} \qquad \textup{a.s.},
\end{equation*}
which proves the optimality of Theorem \ref{tha3}. We remark that it should also be possible to prove Theorem \ref{tha3} without using Theorem \ref{theoremDJ} and Philipp's LIL for the isotropic discrepancy, by deducing it directly from the bounded LIL for empirical processes on VC classes of Alexander and Talagrand \cite{alex}. We conjecture that Theorem \ref{tha3} can be improved to
\begin{equation*}
\limsup_{N \to \infty} \frac{N D(Z_N)}{\sqrt{2 N \log \log N}} = \frac{1}{2} \qquad \textup{a.s.},
\end{equation*}
but this seems to be very difficult to prove.

\section{Point sets with small isotropic discrepancy} \label{sec:iso}

In this section we investigate the isotropic discrepancy of $(0,m,2)$-nets and Fibonacci lattices. In particular we show that the isotropic discrepancy of those point sets converges with order $\mathcal{O}(N^{-1/2})$. Note that the best possible rate of convergence of the isotropic discrepancy is $N^{-2/3} (\log N)^c$ for some $0\le c \le 4$, see \cite{Be1988} and \cite[p.107]{BeCh1987}. Whether $(0,m,2)$-nets and/or Fibonacci lattices achieve the optimal rate of convergence for the isotropic discrepancy is an open question.

\subsection{Nets and sequences}

We give the definition of $(0,m,2)$-nets in base $b$ in the following.

\begin{defn}
Let $b \ge 2$ and $m \ge 1$ be integers. A point set
$P_{b^m} \subseteq [0,1)^2$ consisting of $b^m$ points is
called a {\em $(0,m,2)$-net in base $b$}, if for all nonnegative integers
$d_1,d_2$ with $d_1 + d_2 = m$, each of the elementary
intervals
\begin{equation*}
\prod_{i=1}^2 \left[\frac{a_i}{b^{d_i}},
\frac{a_i+1}{b^{d_i}}\right), \qquad \text{$0 \leq a_i < b^{d_i}$ ($a_i$ an integer),}
\end{equation*}
contains exactly $1$ point of $P_{b^m}$. 
\end{defn}

It is also possible to construct nested $(0,m,2)$-nets, thereby
obtaining an infinite sequence of points.

\begin{defn} \label{def:digital.sequence}
Let $b \ge 2$ be an integer. A sequence
$\bsx_0,\bsx_1,\ldots \in [0,1)^2$ is called a {\em $(0,2)$-sequence in
base $b$}, if for all $m > 0$ and for all $k \ge 0$, the point set
$\bsx_{k b^m}, \bsx_{kb^m + 1}, \ldots, \bsx_{(k+1)b^m-1}$ is a
$(0,m,2)$-net in base $b$.
\end{defn}

Explicit constructions of $(0,m,2)$-nets and $(0,2)$-sequences are due to Sobol'~\cite{So1967} and Faure~\cite{Fa1982}, see also \cite[Chapter~8]{DiPi2010}.

The following is a special case of an unpublished result due to
Gerhard Larcher. For completeness we include a proof here.
\begin{thm}\label{thm_larcher}
For the isotropic discrepancy $J_N$ of a $(0,m,2)$-net $P_N$ in base
$b$ $(N = b^{m})$ we have
\begin{equation*}
J_N(P_N) \leq 4\sqrt{2} b^{- \lfloor m/2 \rfloor} \leq \frac{4 \sqrt{2b}}{\sqrt{N}}.
\end{equation*}
\end{thm}

\begin{proof}
Let $P_N = \{\bsx_0,\ldots, \bsx_{b^m-1}\}$. Let $k =\left\lfloor m / 2 \right\rfloor$ and consider a subcube  $W$ of $[0,1)^{s}$ of the form
\begin{equation*}
W = \left[ \frac{c_{1}}{b^{k}}, \frac{c_{1}+1}{b^{k}} \right) \times \left[ \frac{c_{2}}{b^{k}}, \frac{c_{2}+1}{b^{k}} \right)
\end{equation*}
with $0 \leq c_{i} < b^{k}$ ($c_i$ an integer) for $i=1,2$. The cube $W$ has
volume $b^{-2k}$ and is the union of $b^{m-2k}$ elementary intervals
of order $m$. Indeed, 
\begin{equation*}
W = \bigcup_{v=0}^{b^{m-2k}-1} \left( \left[\frac{c_{1}}{b^{k}}, \frac{c_{1}+1}{b^{k}} \right) \times \left[ \frac{c_{2}}{b^{k}} + \frac{v}{b^{m-k}},\frac{c_{2}}{b^{k}} + \frac{v+1}{b^{m-k}} \right) \right).
\end{equation*}
So $W$ contains exactly $b^{m-2k}$ points of the net.
The diagonal of $W$ has length $\sqrt{2} / b^{k}$.

Let now $A$ be an arbitrary convex subset of $[0,1]^{2}$. Let
$W^{\circ}$ denote the union of cubes $W$ fully contained in $A$ and
let $\overline{W}$ denote the union of cubes $W$ having non-empty
intersection with $A$ or its boundary. The sets $\overline{W}$ and
$W^{\circ}$ are fair with respect to the net, that is,
\begin{equation*}
\frac{1}{N} \sum_{n=0}^{N-1} 1_{\overline{W}}(\bsx_n) =
\lambda(\overline{W}) \quad \mbox{  and  } \quad
\frac{1}{N}\sum_{n=0}^{N-1} 1_{W^\circ}(\bsx_n) = \lambda(W^\circ).
\end{equation*}

We have
\begin{equation*}
\frac{1}{N} \sum_{n=0}^{N-1} 1_{A}(\bsx_n) - \lambda(A) \le
\frac{1}{N} \sum_{n=0}^{N-1} 1_{\overline{W}}(\bsx_n) -
\lambda(\overline{W}) + \lambda(\overline{W} \setminus A)  =
\lambda(\overline{W} \setminus A)
\end{equation*}
and
\begin{equation*}
\frac{1}{N} \sum_{n=0}^{N-1} 1_{A}(\bsx_n) - \lambda(A) \ge
\frac{1}{N} \sum_{n=0}^{N-1} 1_{W^\circ}(\bsx_n) - \lambda(W^\circ)
- \lambda(A \setminus W^\circ) = -\lambda(A \setminus W^\circ).
\end{equation*}

Since the set $A$ is convex, the length of the boundary of $A$ is at
most the circumference of the unit square, which is $4$. Further we have
\begin{equation*}
\overline{W} \setminus A \subseteq \{\bsx \in [0,1]^2 \setminus A:
\|\bsx-\bsy\| \leq \sqrt{2} b^{-k} \text{ for some $\bsy \in A$} \}
\end{equation*}
and therefore
\begin{equation*}
\lambda(\overline{W} \setminus A) \le \lambda(\{\bsx \in [0,1]^2
\setminus A: \|\bsx-\bsy\| \leq \sqrt{2} b^{-k} \text{ for some $\bsy \in A$} \}) \le 4 \sqrt{2} b^{-k},
\end{equation*}
where the last inequality follows from the fact that the outer boundary of the
enclosing set has length at most 4 (which is the circumference of
the square $[0,1]^2$).
%
Moreover 
\begin{equation*}
A \setminus W^\circ \subseteq \{ \bsx \in A : \|\bsx-\bsy\| \leq \sqrt{2} b^{-k} \text{ for some $\bsy \in [0,1]^2 \setminus A$} \}
\end{equation*}
and therefore
\begin{equation*}
\lambda(A \setminus W^\circ) \le \lambda( \{ \bsx \in A : \|\bsx-\bsy\| \leq \sqrt{2} b^{-k} \text{ for some $\bsy \in [0,1]^2 \setminus A$} \}) \le 4 \sqrt{2} b^{-k},
\end{equation*}
since, by the convexity of $A$, the boundary of $A$ has length at
most $4$ (which is the circumference of the square $[0,1]^2$).

Thus we obtain
\begin{equation*}
\left|\frac{1}{N} \sum_{n=0}^{N-1} 1_{A}(\bsx_n) - \lambda(A)
\right| \le 4 \sqrt{2} b^{-k}
\end{equation*}
and hence the result follows.
\end{proof}

Note that the above result only applies when the number of points $N$ is of the form $N = b^m$ (notice that choosing $m=1$ only yields a trivial result, hence one usually chooses a small base $b$ and a 'large' value of $m$). In the following we give an extension where the number of points can take on arbitrary positive integers.

\begin{thm}
For the isotropic discrepancy $J_N$ of the first $N$ points $P_N = \{\bsx_0,\ldots, \bsx_{N-1}\}$ of a $(0,2)$-sequence in base $b$ we have
\begin{equation*}
J_N(P_N) \leq  4 \sqrt{2} \left( b^2 + b^{3/2} \right) \frac{1}{\sqrt{N}}.
\end{equation*}
\end{thm}

\begin{proof}
Let $N \in \mathbb{N}$ have base $b$ expansion $N = N_0 + N_1 b + \cdots + N_m b^m$. Let $P_N = \{\bsx_0,\ldots, \bsx_{N-1}\}$ denote the first $N$ points of a $(0,2)$-sequence in base $b$. 
For $0 \leq k \leq m$ with $N_k > 0$ and $0 \leq \ell < N_k$ the point set
\begin{equation*}
Q_{k,\ell} = \{\bsx_{N_m b^m + \cdots + N_{k+1} b^{k+1} + \ell b^k}, \ldots, \bsx_{N_m b^m + \cdots + N_{k+1} b^{k+1} + (\ell+1) b^k -1}\}
\end{equation*}
is a $(0,k,2)$-net in base $b$ by Definition~\ref{def:digital.sequence}. Thus $P_N$ is a disjoint union of such $(0,k,2)$-nets 
\begin{equation*}
P_N = \bigcup_{\substack{0 \le k \le m \\ N_k > 0}} \bigcup_{0 \le \ell < N_k} Q_{k,\ell}.
\end{equation*}
We have the following triangle inequality for the isotropic discrepancy (which is an analogue to the triangle inequality for the star-discrepancy \cite[p. 115, Theorem~2.6]{KuNi1974})
\begin{equation*}
J_N(P_N) \leq \sum_{\substack{k = 0 \\ N_k > 0}}^m \sum_{\ell =0}^{N_k-1} \frac{b^k}{N} J_{b^k}(Q_{k,\ell}).
\end{equation*}
This inequality holds, since for a spherical cap $C$ we have
\begin{equation*}
\frac{1}{N} \sum_{n=0}^{N-1} 1_C(\bsx_n) - \lambda(C) = \sum_{\substack{k = 0 \\ N_k > 0}}^m \sum_{\ell=0}^{N_k-1} \frac{b^k}{N} \left(\frac{1}{b^k} \sum_{\bsx \in Q_{k,\ell}} 1_C(\bsx) - \lambda(C) \right).
\end{equation*}
Thus we can use Theorem~\ref{thm_larcher} to obtain
\begin{align*}
J_N(P_N) 
&\le \sum_{\substack{k = 0 \\ N_k > 0}}^m \sum_{\ell = 0}^{N_k-1} \frac{b^k}{N} 4 \sqrt{2} b^{-\lfloor k/2\rfloor} = 4 \sqrt{2} \sum_{\substack{k = 0 \\ N_k > 0}}^m \sum_{\ell = 0}^{N_k-1} \frac{b^{\lceil k/2 \rceil}}{N} = 4 \sqrt{2} \frac{\sum_{k = 0}^m N_k b^{\lceil k/2 \rceil}}{\sum_{k = 0}^m N_k b^k} \\
&\le 4 \sqrt{2} (b-1) \sum_{k=0}^m \frac{b^{\lceil k/2 \rceil}}{b^m} \le 4 \sqrt{2} \frac{b(b-1)}{\sqrt{b}-1} b^{-m/2} \le 4 \sqrt{2} \frac{b^{3/2}(b-1)}{\sqrt{b}-1} \frac{1}{\sqrt{N}}.
\end{align*}
The estimate follows from the identity $(a-1)(a+1) = a^2 - 1$.
\end{proof}

\begin{cor}\label{cor1}
{\ }
\begin{enumerate}
\item Let $P_N$ be a $(0,m,2)$-net in base $b$ and let $Z_N = \bsPhi(P_N)
\subseteq \mathbb{S}^2$. Then the spherical cap discrepancy $D(Z_N)$ is bounded by
\begin{equation*}
D(Z_N) \le 44 \sqrt{2} b^{-\lfloor m/2 \rfloor}.
\end{equation*}
\item Let $P_N$ be the first $N$ points of a $(0,2)$-sequence in base $b$ and let $Z_N = \bsPhi(P_N) \subseteq \mathbb{S}^2$. Then the spherical cap discrepancy $D(Z_N)$ is bounded by
\begin{equation*}
D(Z_N) \le 44 \sqrt{2} \left( b^2 + b^{3/2} \right) \frac{1}{\sqrt{N}} \quad \text{for all $N$.}
\end{equation*}
\end{enumerate}
\end{cor}

Note that Item \textit{(2)} improves upon Theorem~\ref{tha3} by a factor of $\sqrt{\log \log N}$ and hence, asymptotically, spherical digital sequences are better than random sequences almost surely.

\begin{table}[ht]
\begin{center}
\begin{tabular}{l|l|l|l|l|l|l|l|l}
$m$ & $6$ & $7$ & $8$ & $9$ & $10$ & $11$ & $12$ & $13$ \\ 
$N = 2^m$ & $64$ &  $128$ & $256$ & $512$ & $1024$ & $2048$ & $4096$ & $8192$ \\ \hline $\frac{\tilde{D}(Z_{N})*N^{3/4}}{\sqrt{\log N}}$ & $0.8829$ & $0.8436$ & $0.8279$ & $0.8632$ & $0.8518$ & $1.2128$ & $1.2285$ & $0.9546$ \\ $\frac{\tilde{D}(Z_{N})*N^{3/4}}{\log N}$ & $0.4329$ & $0.3829$ & $0.3515$ & $0.3456$ & $0.3235$ & $0.4392$ & $0.4259$ & $0.3180$ \\ \hline\hline
$m$ & $14$ & $15$ & $16$ & $17$ & $18$ & $19$ & $20$ & $21$ \\ 
$N=2^m$ & $16384$ &  $32768$ & $65536$ & $131072$ & $262144$ & $524288$ & $1048576$ & $2097152$ \\ \hline  $\frac{\tilde{D}(Z_{N})*N^{3/4}}{\sqrt{\log N}}$  & $0.7925$ & $0.8862$ & $1.0331$ & $0.8337$ & $0.8562$ & $0.9854$ & $1.1167$ & $1.1463$ \\ $\frac{\tilde{D}(Z_{N})*N^{3/4}}{\log N}$  & $0.2544$ & $0.2748$ & $0.3102$ & $0.2428$ & $0.2424$ & $0.2715$ & $0.2999$ & $0.3004$
\end{tabular}
\end{center}
\caption{The value $\tilde{D}(Z_N)$ denotes the maximum of the absolute values of the local discrepancies of $Z_N$ at the spherical caps centered at the spherical digital net $Z_N$ based on a two-dimensional Sobol' point set.}
\label{table2}
\end{table}

The numerical experiments shown in Table~\ref{table2} seem to suggest that the correct order of the spherical cap discrepancy of spherical digital nets is 
\begin{equation*}
\frac{(\log N)^c}{N^{3/4}} \quad \mbox{for some } 1/2 \le c \le 1.
\end{equation*}
In those experiments we calculated the spherical cap discrepancy for all spherical caps centered at a spherical digital net based on a two-dimensional Sobol' point set (this of course gives only a lower bound).

\subsection{Fibonacci lattices}

The Fibonacci numbers $F_m$ are given by $F_1=1, F_2=1$ and $F_m = F_{m-1}+F_{m-2}$ for all $m > 2$. A Fibonacci lattice is a point set of $F_m$ points in $[0,1)^2$ given by
\begin{equation*}
\boldsymbol{f}_m := \left(\frac{n}{F_m}, \left\{ \frac{n
F_{m-1}}{F_m} \right\} \right), \quad 0 \le n < F_m,
\end{equation*}
where $\{x\} = x- \lfloor x \rfloor$ denotes the fractional part for nonnegative real numbers
$x$. The set
\begin{equation*}
\mathcal{F}_m :=\{\boldsymbol{f}_0,\ldots, \boldsymbol{f}_{F_m-1}\}
\end{equation*}
is called a Fibonacci lattice point set.

The spherical Fibonacci lattice points are then given by
\begin{equation*}
\boldsymbol{z}_n = \bsPhi(\boldsymbol{f}_n), \quad 0 \le n < F_m,
\end{equation*}
and the point set
\begin{equation*}
Z_{F_m} = \{\boldsymbol{z}_0,\ldots, \boldsymbol{z}_{F_m-1}\}
\end{equation*}
is the spherical Fibonacci lattice point set.

In the following we prove a bound on the isotropic discrepancy of Fibonacci lattices, see also \cite{La85, La88, La91}.

\begin{lem}
For the isotropic discrepancy $J_{F_m}$ of a Fibonacci lattice
$\mathcal{F}_{m}$ we have
\begin{equation*}
J_{F_m}(\mathcal{F}_m) \leq \begin{cases}
4 \sqrt{ 2 / F_{m}} & \text{if $m$ is odd,} \\
4 \sqrt{ 8 / F_{m}} & \text{if $m$ is even.}
\end{cases}
\end{equation*}
\end{lem}

\begin{proof}
Consider the case of odd integers $m$ first. From
\cite[Theorem~3]{NiSl94} it follows that for $m \in \mathbb{N}$ the
Fibonacci lattice $\mathcal{F}_{2m+1}$ can be generated by the
vectors
\begin{equation*}
\boldsymbol{a}_{2m+1} = (F_m/F_{2m+1}, (-1)^{m-1} F_{m+1}/F_{2m+1}),
\quad \boldsymbol{b}_{2m+1} = (F_{m+1}/F_{2m+1}, (-1)^m
F_m/F_{2m+1}).
\end{equation*}
This means that
\begin{equation*}
\mathcal{F}_{2m+1} = \{u \boldsymbol{a}_{2m+1} + v
\boldsymbol{b}_{2m+1} : u, v \in \mathbb{Z} \} \cap [0,1)^2.
\end{equation*}
Let
\begin{equation*}
U(\bsy) = \{\bsy+\bsx \in [0,1]^2: \bsx = s \boldsymbol{a}_{2m+1} +
t \boldsymbol{b}_{2m+1}, 0 \le s, t < 1 \}.
\end{equation*}
We call $U(\boldsymbol{f}_n)$ a unit cell (belonging to the point
$\boldsymbol{f}_n$). Note that the area of a unit cell is
$1/F_{2m+1}$ and each unit cell contains exactly one point of the
lattice, see \cite{NiSl94}.

Since $\boldsymbol{a}_{2m+1} \perp \boldsymbol{b}_{2m+1}$ it follows
that the minimum distance between points of the Fibonacci lattice is
\begin{equation*}
d_{\min}(\mathcal{F}_{2m+1}) = \min\{ \|\boldsymbol{a}_{2m+1}\|,
\|\boldsymbol{b}_{2m+1}\|\} =
\frac{\sqrt{F_m^2+F_{m+1}^2}}{F_{2m+1}} = \frac{1}{\sqrt{F_{2m+1}}}.
\end{equation*}
Thus the diameter of a unit cell is $\sqrt{2/F_{2m+1}}$.

Let now $A$  be an arbitrary convex subset of $[0,1]^2$. Let
$W^\circ$ denote the union of all unit cells fully contained in $A$
and let $\overline{W}$ denote the union of all unit cells with
nonempty intersection with $A$ or its boundary.

We have
\begin{equation*}
\frac{1}{F_{2m+1}} \sum_{n=0}^{F_{2m+1}-1} 1_{A}(\boldsymbol{f}_n) -
\lambda(A) \le \frac{1}{F_{2m+1}} \sum_{n=0}^{F_{2m+1}-1}
1_{\overline{W}}(\boldsymbol{f}_n) - \lambda(\overline{W}) +
\lambda(\overline{W} \setminus A) = \lambda(\overline{W} \setminus
A)
\end{equation*}
and
\begin{equation*}
\frac{1}{F_{2m+1}} \sum_{n=0}^{F_{2m+1}-1} 1_{A}(\boldsymbol{f}_n) -
\lambda(A) \ge \frac{1}{F_{2m+1}} \sum_{n=0}^{F_{2m+1}-1}
1_{W^\circ}(\boldsymbol{f}_n) - \lambda(W^\circ) - \lambda(A
\setminus W^\circ) = - \lambda(A \setminus W^\circ).
\end{equation*}

Since the set $A$ is convex, the length of the boundary of $A$ is at
most the circumference of the unit square, which is $4$. Further we have
\begin{equation*}
\overline{W} \setminus A \subseteq \{\bsx \in [0,1]^2 \setminus A:
\|\bsx-\bsy\| \le \sqrt{2/F_{2m+1}} \text{ for some $\bsy \in A$} \}
\end{equation*}
and therefore
\begin{equation*}
\lambda(\overline{W} \setminus A) \le \lambda(\{\bsx \in [0,1]^2
\setminus A:
\|\bsx-\bsy\| \le \sqrt{2/F_{2m+1}} \text{ for some $\bsy \in A$} \}) \le 4 \sqrt{2/F_{2m+1}},
\end{equation*}
where the last inequality follows since the outer boundary of
enclosing set has length at most 4 (which is the circumference of
the square $[0,1]^2$).

On the other hand
\begin{equation*}
A \setminus W^\circ \subseteq \{\bsx \in
A: \|\bsx-\bsy\| \le \sqrt{2/F_{2m+1}} \text{ for some $\bsy \in [0,1]^2 \setminus A$} \}
\end{equation*}
and therefore
\begin{equation*}
\lambda(A \setminus W^\circ) \le \lambda(\{\bsx \in A: \|\bsx-\bsy\| \le \sqrt{2/F_{2m+1}} \text{ for some $\bsy \in [0,1]^2 \setminus A$} \}) \le 4
\sqrt{2/F_{2m+1}},
\end{equation*}
since, by the convexity of $A$, the boundary of $A$ has length at
most $4$ (which is the circumference of the square $[0,1]^2$).

Thus we obtain
\begin{equation*}
\left|\frac{1}{F_{2m+1}} \sum_{n=0}^{F_{2m+1} -1}
1_{A}(\boldsymbol{f}_n) - \lambda(A) \right| \le 4
\sqrt{\frac{2}{F_{2m+1}}}.
\end{equation*}

Now we consider even integers $n = 2m$ with $m \ge 2$. Using the
identity $F_m F_{2m-1} - F_{m-1} F_{2m} = (-1)^{m-1} F_m$, we obtain
$F_m F_{2m-1} \equiv (-1)^{m-1} F_m \pmod{F_{2m}}$. 
Consequently
\begin{equation*}
\left(\frac{k}{F_{2m}}, \left\{\frac{k
F_{2m-1}}{F_{2m}}\right\}\right) = \left(\frac{F_m}{F_{2m}},
\left\{\frac{(-1)^{m-1} F_m}{F_{2m}} \right\}\right) \qquad \text{for $k = F_m$.}
\end{equation*}
Thus the Fibonacci lattice has the equivalent generating vector
\begin{equation*}
\boldsymbol{a}_{2m} = \left(\frac{F_m}{F_{2m}}, \frac{(-1)^{m-1}
F_m}{F_{2m}}\right).
\end{equation*}
Analogously, using the equality $F_{m+1} F_{2m-1} - F_m F_{2m} =
(-1)^m F_{m-1}$, we obtain the generating vector
\begin{equation*}
\boldsymbol{b}_{2m} = \left(\frac{F_{m+1}}{F_{2m}}, \frac{(-1)^m
F_{m-1}}{F_{2m}}\right).
\end{equation*}

The area of the parallelogram spanned by $\boldsymbol{a}_{2m}$ and
$\boldsymbol{b}_{2m}$ is
\begin{equation*}
\left|\det\left(\begin{array}{rr} F_m/F_{2m} & (-1)^{m-1} F_m/F_{2m}
\\ F_{m+1}/F_{2m} & (-1)^m F_{m-1}/F_{2m} \end{array} \right)\right|
= \frac{F_m F_{m-1} + F_{m} F_{m+1}}{F_{2m}^2} = \frac{1}{F_{2m}}.
\end{equation*}
Thus the parallelogram spanned by $\boldsymbol{a}_{2m}$ and
$\boldsymbol{b}_{2m}$ does not contain any point of the Fibonacci
lattice in its interior (i.e. is a unit cell of the Fibonacci
lattice,  see \cite{Cassels59, NiSl94}). Thus we have
\begin{equation*}
\mathcal{F}_{2m} = \{u \boldsymbol{a}_{2m} + v \boldsymbol{b}_{2m} :
u,v \in \mathbb{Z}\} \cap [0,1)^2.
\end{equation*}

Let
\begin{equation*}
U(\bsy) = \{\bsy+\bsx \in [0,1]^2: \bsx = s \boldsymbol{a}_{2m+1} +
t \boldsymbol{b}_{2m+1}, 0 \le s, t < 1 \}.
\end{equation*}
We call $U(\boldsymbol{f}_n)$ a unit cell (belonging to the point
$\boldsymbol{f}_n$). Note that the area of a unit cell is
$1/F_{2m+1}$ and each unit cell contains exactly one point of the
lattice.

Now
\begin{equation*}
\|\boldsymbol{a}_{2m}\|^2 = \frac{2 F_m^2}{F_{2m}^2}
\end{equation*}
and
\begin{equation*}
\|\boldsymbol{b}_{2m}\|^2 = \frac{F_{m+1}^2 + F_{m-1}^2}{F_{2m}^2} >
\frac{F_m^2 + F_m (2F_{m-1})}{F^2_{2m}} > \frac{2F_m^2}{F^2_{2m}} =
\|\boldsymbol{a}_{2m}\|^2.
\end{equation*}
Further it can be checked that $\|\boldsymbol{a}_{2m}+
\boldsymbol{b}_{2m}\|, \|\boldsymbol{a}_{2m}-\boldsymbol{b}_{2m}\| >
\|\boldsymbol{a}_{2m}\|$. Thus the minimum distance between points of the Fibonacci lattice is
\begin{equation*}
d_{\min}(\mathcal{F}_{2m}) = \|\boldsymbol{a}_{2m}\| =
\frac{\sqrt{2}F_m}{F_{2m}}.
\end{equation*}

The diameter of a unit cell is given by $\|\boldsymbol{a}_{2m} +
\boldsymbol{b}_{2m}\|$. Using the relations $F_m = F_{m-1} +
F_{m-2}$, $F_m^2 +F_{m-1}^2 = F_{2m-1}$ and $F_{2m} = (2F_{m-1}
+F_m) F_m$ we obtain
\begin{equation*}
F_{2m}^2 \|\boldsymbol{a}_{2m} + \boldsymbol{b}_{2m}\|^2 =
(F_m+F_{m+1})^2 + (F_m-F_{m-1})^2 = 2F_{2m} + 4 F_{2m-1} + 2F_m^2 <
8 F_{2m}.
\end{equation*}
Thus the diameter of a unit cell is bounded by
\begin{equation*}
\|\boldsymbol{a}_{2m} + \boldsymbol{b}_{2m}\| \le
\sqrt{\frac{8}{F_{2m}}}.
\end{equation*}
The result now follows by using the same arguments as in the
previous case.
\end{proof}

\begin{cor}\label{cor2}
Let $\mathcal{F}_{m}$ be a Fibonacci lattice and let $Z_{F_m} =
\bsPhi(\mathcal{F}_{m}) \subseteq \mathbb{S}^2$. Then the spherical cap discrepancy $D(Z_{F_m})$ is bounded by
\begin{equation*}
D(Z_{F_m}) \leq
\begin{cases}
44 \sqrt{ 2 / F_{m}} & \text{if $m$ is odd,} \\
44 \sqrt{ 8 / F_{m}} & \text{if $m$ is even.}
\end{cases}
\end{equation*}
\end{cor}

\begin{table}[ht]
\begin{center}
\begin{tabular}{l|l|l|l|l|l|l|l|l}
$m$ & $17$ & $18$ & $19$ & $20$ & $21$ & $22$ & $23$ & $24$ \\ 
$F_m$ & $1597$ &  $2584$ & $4181$ & $6765$ & $10946$ & $17711$ & $28657$ & $46368$ \\ \hline
$\frac{\tilde{D}(Z_{F_m})*F_m^{3/4}}{\sqrt{\log F_m}}$ & $0.6729$ & $0.6373$ & $0.6228$ & $0.6661$ & $0.6953$ & $0.6890$ & $0.7427$ & $0.6900$ \\ $\frac{\tilde{D}(Z_{F_m})*F_m^{3/4}}{\log F_m}$ & $0.2477$ & $0.2273$ & $0.2156$ & $0.2243$ & $0.2279$ & $0.2203$ & $0.2318$ & $0.2105$ \\ \hline\hline
$m$ & $25$ & $26$ & $27$ & $28$ & $29$ & $30$ & $31$ & $32$ \\ 
$F_m$ & $75025$ &  $121393$ & $196418$ & $317811$ & $514229$ & $832040$ & $1346269$ & $2178309$ \\ \hline $\frac{\tilde{D}(Z_{F_m})*F_m^{3/4}}{\sqrt{\log F_m}}$ & $0.6957$ & $0.7249$ & $0.7531$ & $0.7205$ & $0.8562$ & $0.7455$ & $0.7862$ & $0.8082$ \\ $\frac{\tilde{D}(Z_{F_m})*F_m^{3/4}}{\log F_m}$ & $0.2076$ & $0.2118$ & $0.2157$ & $0.2024$ & $0.2361$ & $0.2019$ & $0.2092$ & $0.2115$
\end{tabular}
\end{center}
\caption{The value $\tilde{D}(Z_{F_m})$ denotes the maximum of the absolute values of the local discrepancies of $Z_{F_m}$ at the spherical caps centered at the spherical Fibonacci points $Z_{F_m}$.}
\label{table1}
\end{table}

The numerical experiments shown in Table~\ref{table1} seem to suggest that the correct order of the spherical cap discrepancy of spherical Fibonacci lattice points is of order
\begin{equation*}
\frac{(\log F_m)^c}{F_m^{3/4}} \quad \mbox{for some } 1/2 \le c \le 1.
\end{equation*}
In those experiments we calculated the spherical cap discrepancy for all spherical caps centered at the spherical Fibonacci points (this of course gives only a lower bound).

\section{Level curves of the distance function and their properties}
The Euclidean distance of two points on $\mathbb{S}^2$, given by $\bsw = \bsPhi(u,v)$ and $\bsz = \bsPhi(\alpha,\tau)$, can be written as
\begin{equation*}
\left\| \bsw - \bsz \right\|^2 = 2 \left( 1 - \bsw \cdot \bsz \right) = 2 \Big[ \underbrace{1 - \left( 1 - 2v \right) \left( 1 - 2 \tau \right)}_{2 \left( v + \tau - 2 v \tau \right)} - 4 \sqrt{\left( 1 - v \right) v \left( 1 - \tau \right) \tau} \, \cos( 2 \pi ( u - \alpha ) ) \Big].
\end{equation*}
The well-defined boundary curve of a spherical cap $C(\bsw, t)$ has the implicit representation
\begin{equation} \label{eq:implicit.representation}
F(\alpha, \tau ) \DEF \left\| \bsw - \bsz \right\|^2 - 2 \left( 1 - t \right) = 0, \qquad \text{where $\bsz = \bsPhi(\alpha,\tau)$ moves on $\partial C(\bsw, t)$.}
\end{equation}
(Here, $\sqrt{2(1-t)}$ is the distance from the centre $\bsw$ to a point on the boundary of $C(\bsw, t)$.) Relation \eqref{eq:implicit.representation} describes a level curve $\mathcal{C}_{\bsw}$ of the distance function $\| \bsw - \bsPhi(\alpha,\tau) \|$ (for $\bsw$ fixed) in the parameter space which is the unit square, cf. Figure~\ref{fig:curvature}. For further references we record that for each $\bsw$ there are, in general, two exceptional levels 
\begin{align*}
r_{\bsw}^2 &= 2 \left( 1 - t_{\bsw} \right) = \left\| \bsw - \bsp \right\|^2 = 4 v, \\ 
\rho_{\bsw}^2 &= 2 \left( 1 - t_{\bsw}^\prime \right) = 2 \left( 1 + t_{\bsw} \right) = \left\| \bsw + \bsp \right\|^2 = 4 \left( 1 - v \right),
\end{align*}
where the boundary of the spherical cap centred at $\bsw$ passes through the North Pole ($\bsp$) and the South Pole ($-\bsp$), respectively, which may coincide if $\bsw$ is on the equator. For these level curves the singular behaviour at the poles imposed by the parametrisation $\bsPhi$ plays a role. 

Suppose $u=1/2$. Because the sign of the difference $u-\alpha$ is absorbed by the cosine function in the distance function, a level curve (a level set) is symmetric with respect to the vertical line $\alpha = u$. The shape of the curves (sets) do not change when the point $\bsw$ is rotated about the polar axis except a part moving outside the left side of the unit square enters at the right side (``wrap around''). This ``modulo $1$'' behaviour complicates considerations regarding convexity of the pre-image of a spherical cap centred at $\bsw$ under the distance function. Similarly and in the same sense, the level curves (sets) are symmetric with respect to the vertical line $\alpha = u \pm 1/2 \mod 1$ which passes through the parameter point of the antipodal point $-\bsw = \bsPhi( u \pm 1/2 \mod 1, 1 - v)$. When identifying the left and right side of the unit square $[0,1]^2$, we get the ``cylindrical view''. Clearly, all level curves except the critical ones associated with the distance to one of the poles are closed on the open cylinder. Thus, the critical level curves separate the open cylinder into three parts corresponding to the cases when neither pole is contained in the spherical cap centred at $\bsw$, only one pole is contained in the spherical cap, and both poles are contained in the spherical cap. In both the first and last case the level curve can not escape the level set bounded by a critical curve (and the boundary of the cylinder). It is closed even in the unit square provided the level set is contained in its interior. In the middle case the level curves wrap around the cylinder; that is, start at the left side of the unit square and end at its right side at the same height; cf. Figure~\ref{fig:curvature}.

{\bf Let $\bsw$ be the North or the South Pole.} Then the level curves are horizontal lines in the unit square which are smooth curves. Moreover, the pre-image of a spherical cap centred at one of these poles is a convex set (a rectangle).

{\bf Let $\bsw$ be different from either pole (that is $0 < v < 1$).} We use the signed curvature of the implicitly given level curve to determine the segments where it is convex (concave). First, we collect the partial derivatives up to second order:
\begin{subequations} \label{eq:partial.derivatives}
\begin{align}
F_{\alpha}(\alpha, \tau ) &= - 16 \pi \sqrt{\left( 1 - v \right) v \left( 1 - \tau \right) \tau} \, \sin( 2 \pi ( u - \alpha ) ), \label{eq:Fa} \\
F_{\tau}(\alpha, \tau ) &= 4 \left( 1 - 2 v - \left( 1 - 2 \tau \right) \frac{\sqrt{\left( 1 - v \right) v}}{\sqrt{\left( 1 - \tau \right) \tau}} \, \cos( 2 \pi ( u - \alpha ) ) \right), \label{eq:Ft} \\
F_{\alpha\alpha}(\alpha, \tau ) &= 32 \pi^2 \sqrt{\left( 1 - v \right) v \left( 1 - \tau \right) \tau} \, \cos( 2 \pi ( u - \alpha ) ), \label{eq:Faa} \\
F_{\alpha\tau}(\alpha, \tau ) &= - 8 \pi \left( 1 - 2 \tau \right) \frac{\sqrt{\left( 1 - v \right) v}}{\sqrt{\left( 1 - \tau \right) \tau}} \, \sin( 2 \pi ( u - \alpha ) ), \label{eq:Fat} \\
F_{\tau\tau}(\alpha, \tau ) &= \frac{2}{\left( 1 - \tau \right) \tau} \frac{\sqrt{\left( 1 - v \right) v}}{\sqrt{\left( 1 - \tau \right) \tau}} \, \cos( 2 \pi ( u - \alpha ) ).  \label{eq:Ftt}
\end{align}
\end{subequations}

We observe that the partial derivatives involving differentiation with respect to $\tau$ become singular as $\tau$ approaches $0$ (North Pole) or $1$ (South Pole). 

The signed curvature at a point $(\alpha, \tau)$ of $\mathcal{C}_{\bsw}$ is given by
\begin{align} \label{eq:curvature}
\kappa 
&= \kappa(\alpha, \tau) = \frac{F_{\alpha\alpha} F_{\tau}^2 - 2 F_{\alpha \tau} F_{\alpha} F_{\tau} + F_{\tau\tau} F_{\alpha}^2}{\left( F_{\alpha}^2 + F_{\tau}^2 \right)^{3/2}}.
\end{align}

First, we discuss the denominator. Substituting the relations \eqref{eq:Fa} and \eqref{eq:Ft} we obtain
\begin{equation*}
\begin{split}
F_{\alpha}^2 + F_{\tau}^2 
&= 16 \left( 1 - 2 v - \left( 1 - 2 \tau \right) \frac{\sqrt{\left( 1 - v \right) v}}{\sqrt{\left( 1 - \tau \right) \tau}} \cos( 2 \pi ( u - \alpha ) ) \right)^2 \\
&\phantom{=}+ 16^2 \pi^2 \left( 1 - v \right) v \left( 1 - \tau \right) \tau \left[ \sin( 2 \pi ( u - \alpha ) ) \right]^2.
\end{split}
\end{equation*}
A necessary condition for the vanishing of $F_{\alpha}^2 + F_{\tau}^2$ is $\sin( 2 \pi ( u - \alpha ) ) = 0$; that is, either $\alpha = u$ or $\alpha = u \pm 1/2 \mod 1$. In the first case one has $\cos( 2 \pi ( u - \alpha ) ) = 1$ and, therefore, $F_{\alpha}^2 + F_{\tau}^2 = 0$ if and only if $\tau = v$. In the second case one has $\cos( 2 \pi ( u - \alpha ) ) = -1$ and, therefore, $F_{\alpha}^2 + F_{\tau}^2 = 0$ if and only if $\tau = 1-v$. It follows that the denominator of the curvature formula \eqref{eq:curvature} vanishes if and only if $\bsz = \bsPhi(\alpha,\tau)$ (on the boundary of the spherical cap centred at $w = \bsPhi(u,v)$) coincides with $\bsw$ (that is the spherical cap degenerates to the point $\bsw$) or $\bsz$ coincides with the antipodal point of $\bsw$ (that is the closed spherical cap is the whole sphere). In either of these cases the pre-image of the spherical cap under $\bsPhi$ is convex.

Suppose that the {\bf spherical cap is neither a point nor the whole sphere}. Then $F_{\alpha}^2 + F_{\tau}^2 \neq 0$. For the numerator in \eqref{eq:curvature} we obtain after some simplifications
\begin{align*}
& F_{\alpha\alpha} F_{\tau}^2 - 2 F_{\alpha \tau} F_{\alpha} F_{\tau} + F_{\tau\tau} F_{\alpha}^2 \\ & = 512 \pi^2 \sqrt{\left( 1 - v \right) v \left( 1 - \tau \right) \tau} \left( 1 - 2 v - \left( 1 - 2 \tau \right) \frac{\sqrt{\left( 1 - v \right) v}}{\sqrt{\left( 1 - \tau \right) \tau}} \cos( 2 \pi (u - \alpha) ) \right)^2 \cos( 2 \pi (u - \alpha) ) \\
&\phantom{=}+ 512 \pi^2 \left( 1 - v \right) v \frac{\sqrt{\left( 1 - v \right) v}}{\sqrt{\left( 1 - \tau \right) \tau}} \cos( 2 \pi (u - \alpha) ) \left[ \sin( 2 \pi (u - \alpha) ) \right]^2 \\
&\phantom{=}- 1024 \pi^2 \left( 1 - v \right) v \left( 1 - 2 \tau \right) \left( 1 - 2 v - \left( 1 - 2 \tau \right) \frac{\sqrt{\left( 1 - v \right) v}}{\sqrt{\left( 1 - \tau \right) \tau}} \cos( 2 \pi (u - \alpha) ) \right) \left[ \sin( 2 \pi (u - \alpha) ) \right]^2.
\end{align*}
The right-hand side above can be written as a polynomial in $x = \cos( 2 \pi (u - \alpha) )$ as follows
\begin{equation*}
\begin{split}
F_{\alpha\alpha} F_{\tau}^2 - 2 F_{\alpha \tau} F_{\alpha} F_{\tau} + F_{\tau\tau} F_{\alpha}^2 
&= A x \left( 1 - 2 v - B  H x \right)^2 + \frac{A^2 B x \left( 1 - x^2 \right)}{512 \pi^2 \left( 1 - \tau \right) \tau} \\
&\phantom{=}- 2 A B H \left( 1 - 2v - B H x \right) \left( 1 - x^2 \right),
\end{split}
\end{equation*}
where 
\begin{equation*}
A = 512 \pi^2 \sqrt{\left( 1 - v \right) v \left( 1 - \tau \right) \tau}, \qquad B = \frac{\sqrt{\left( 1 - v \right) v}}{\sqrt{\left( 1 - \tau \right) \tau}}, \qquad H = 1 - 2 \tau.
\end{equation*}
Reordering with respect to falling powers of $x$, we observe that the coefficient of $x^2$ vanishes and after simplifications we arrive at
\begin{equation*}
\begin{split}
F_{\alpha\alpha} F_{\tau}^2 &- 2 F_{\alpha \tau} F_{\alpha} F_{\tau} + F_{\tau\tau} F_{\alpha}^2 = - A B \left( B H^2 + \frac{A}{512 \pi^2 \left( 1 - \tau \right) \tau} \right) x^3 \\
&\phantom{=}+ A \left( 2 B^2 H + \left( 1 - 2v \right)^2+ \frac{A B}{512 \pi^2 \left( 1 - \tau \right) \tau} \right) x - 2 A B H \left( 1 - 2 v \right).
\end{split}
\end{equation*}
%
The coefficient of $x^3$ does not vanish for $0 < v < 1$ and $0<\tau<1$. Hence, we divide and get
\begin{equation*}
- \frac{F_{\alpha\alpha} F_{\tau}^2 - 2 F_{\alpha \tau} F_{\alpha} F_{\tau} + F_{\tau\tau} F_{\alpha}^2}{A B \left( B H^2 + \frac{A}{512 \pi^2 \left( 1 - \tau \right) \tau} \right)} = x^3 + p \, x + q \FED Q(\tau; x) \FED Q(x),
\end{equation*}
where (using the definitions of $A$, $B$, and $H$)
\begin{equation*}
p = p(\tau) = - \frac{\left( 1 - 2 v \right)^2 + B^2 \left( 1 + 2 H^2 \right)}{B^2 \left( 1 + H^2 \right)}, \qquad q = q(\tau) = \frac{2 \left( 1 - 2 v \right) \left( 1 - 2 \tau \right)}{B \left( 1 + H^2 \right)}.
\end{equation*}
We observe that $p(1-\tau) = p(\tau)$ and $q(1-\tau) = - q(\tau)$. Hence, $Q(\tau; x) = - Q(1-\tau; -x)$ for all $x$. In particular, if $\xi$ is a zero of $Q(\tau; \cdot)$, then so is $-\xi$ a zero of $Q(1-\tau; \cdot)$ and vice versa. The monic polynomial $Q$ of degree $3$ with real coefficients has either one or three real solutions (counting multiplicity). 
\MARKED{Black}{With the help of Mathematica we find that the discriminant of the polynomial $Q$ is positive,
\begin{equation*}
\discr(Q) = - 4 p^3 - 27 q^2 > 0;
\end{equation*}
that is, the polynomial $Q$ has three distinct real roots.
}
\begin{figure}[ht]
\begin{center} 
\includegraphics[scale=.05]{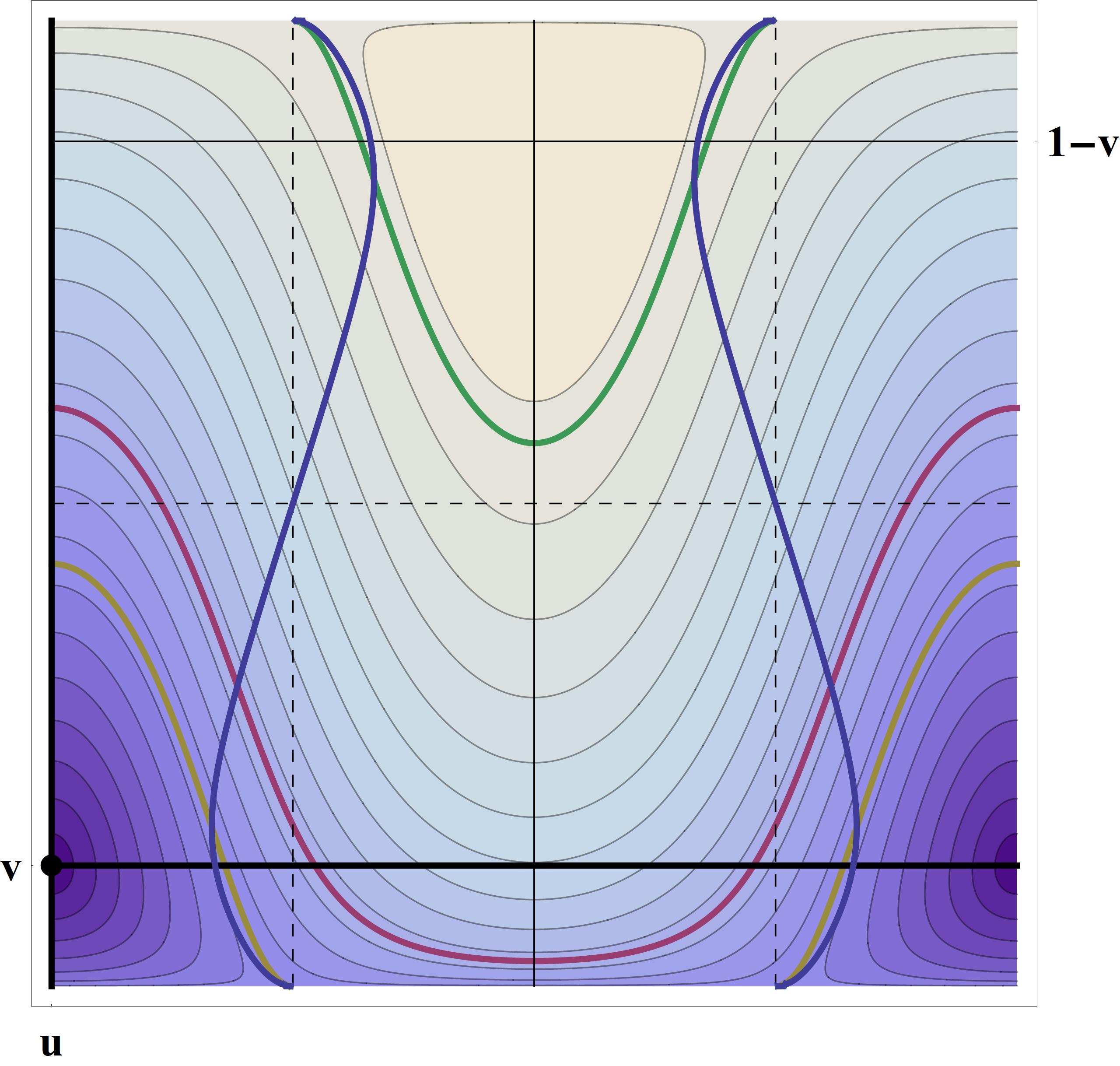}
\includegraphics[scale=.05]{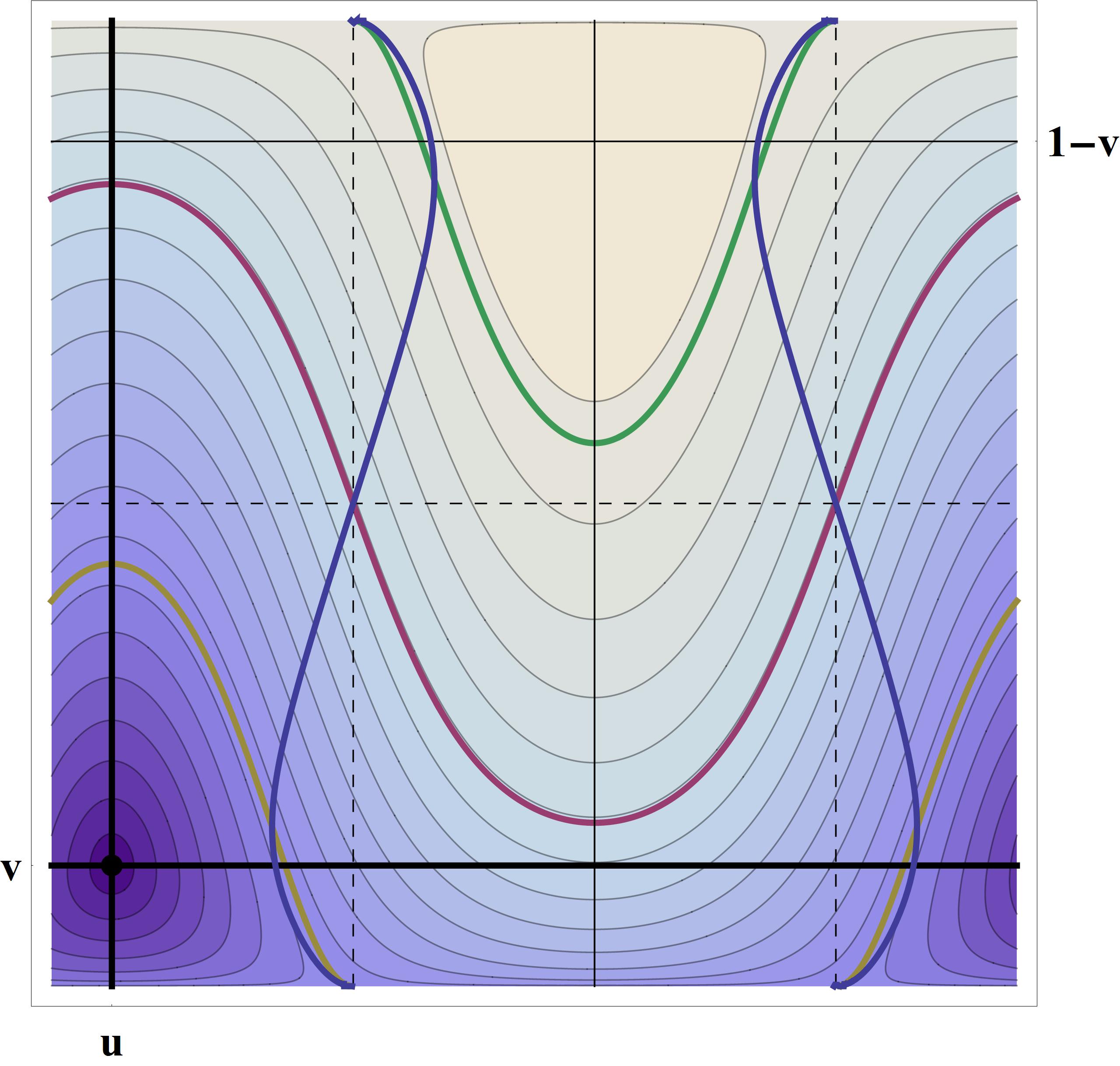}
\includegraphics[scale=.05]{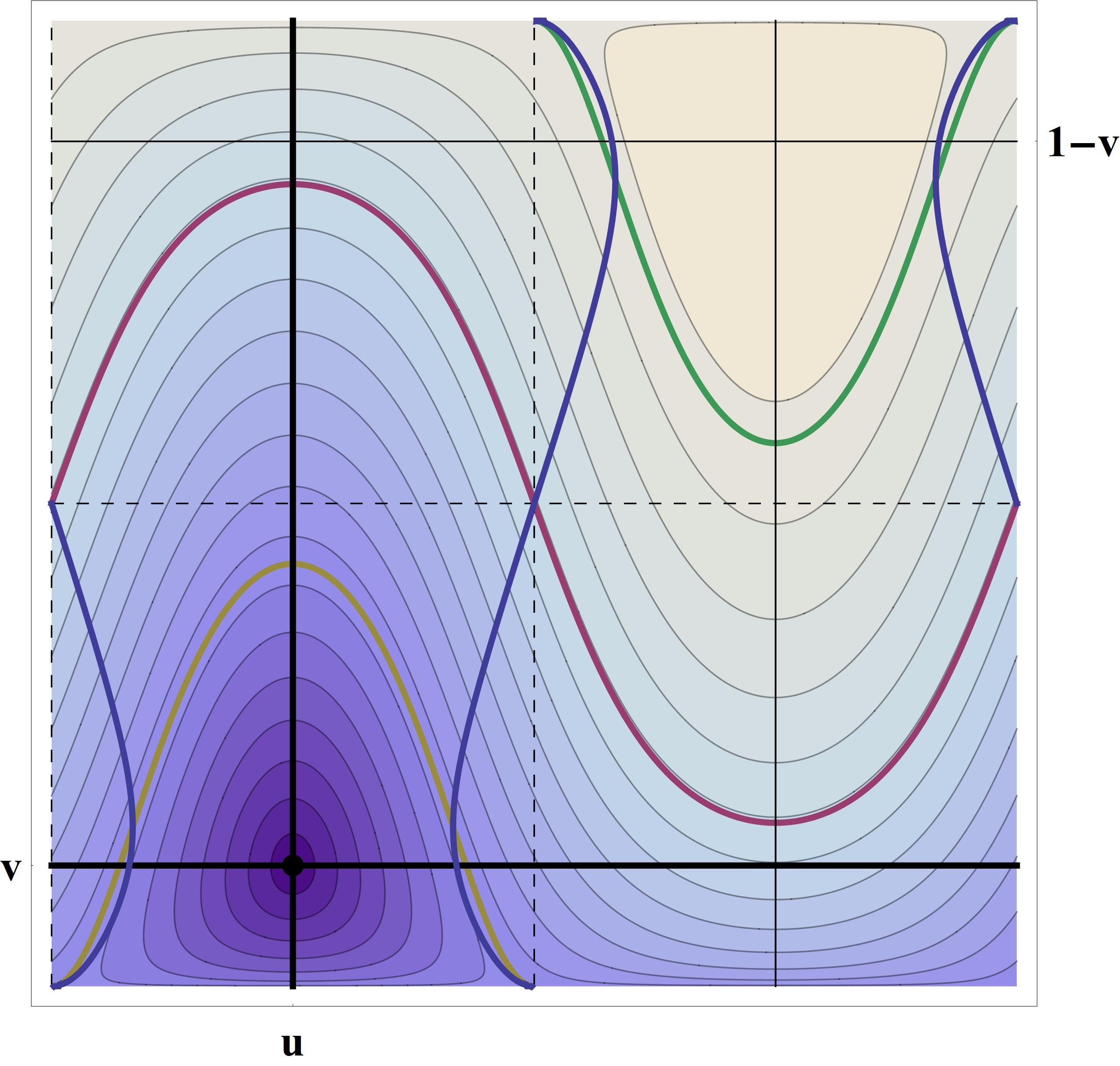}
\includegraphics[scale=.05]{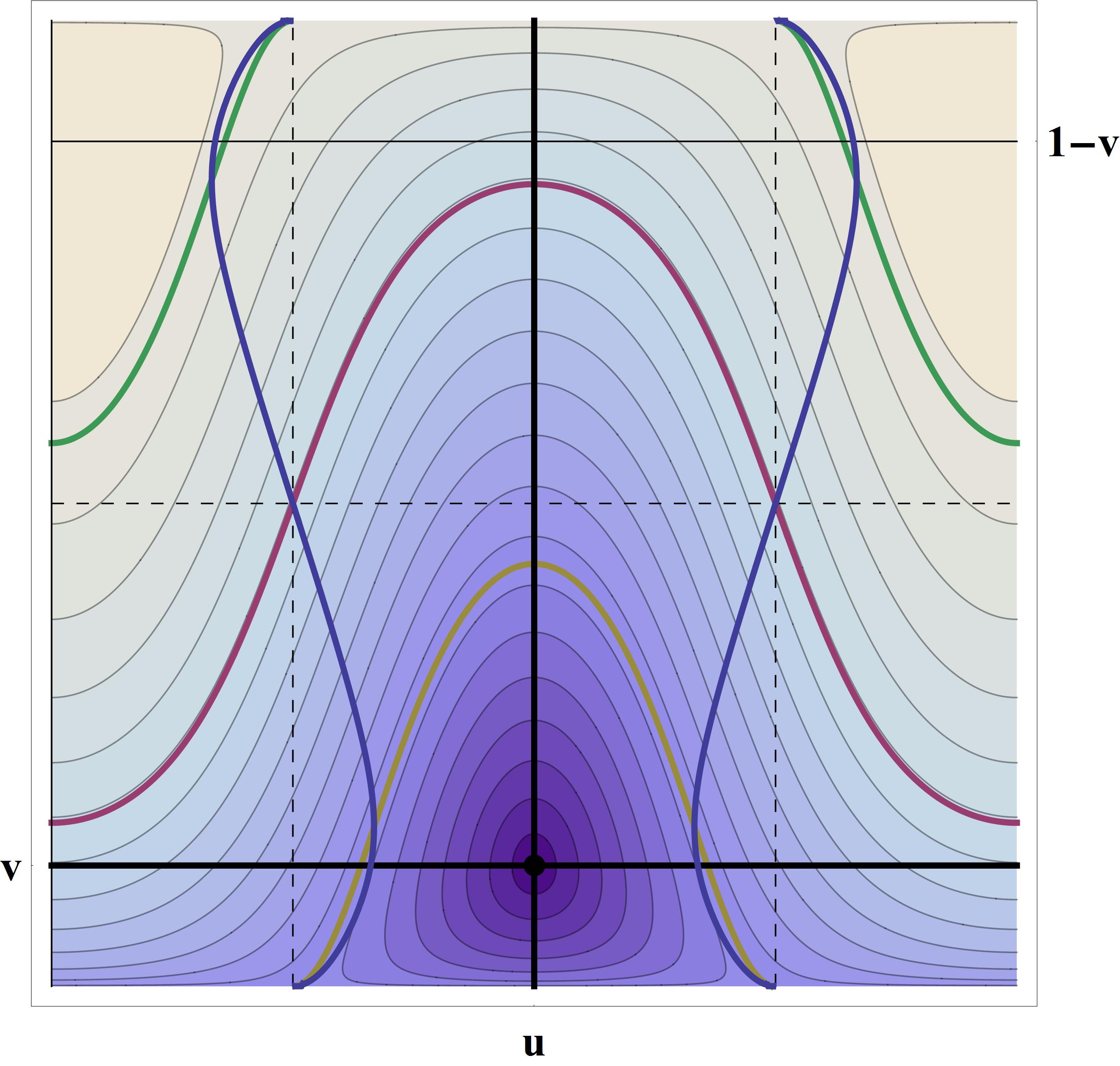}
\includegraphics[scale=.05]{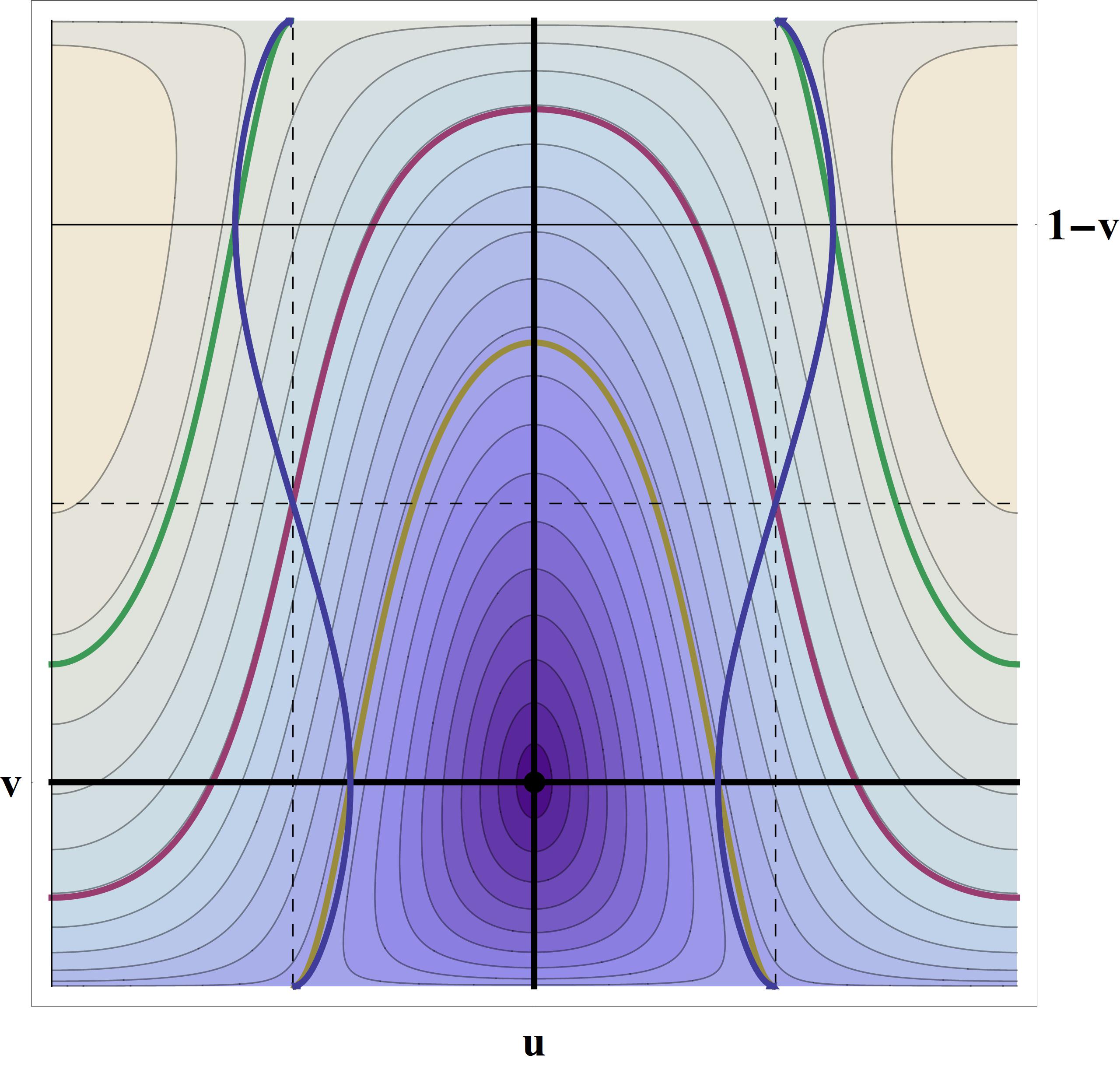}
\includegraphics[scale=.05]{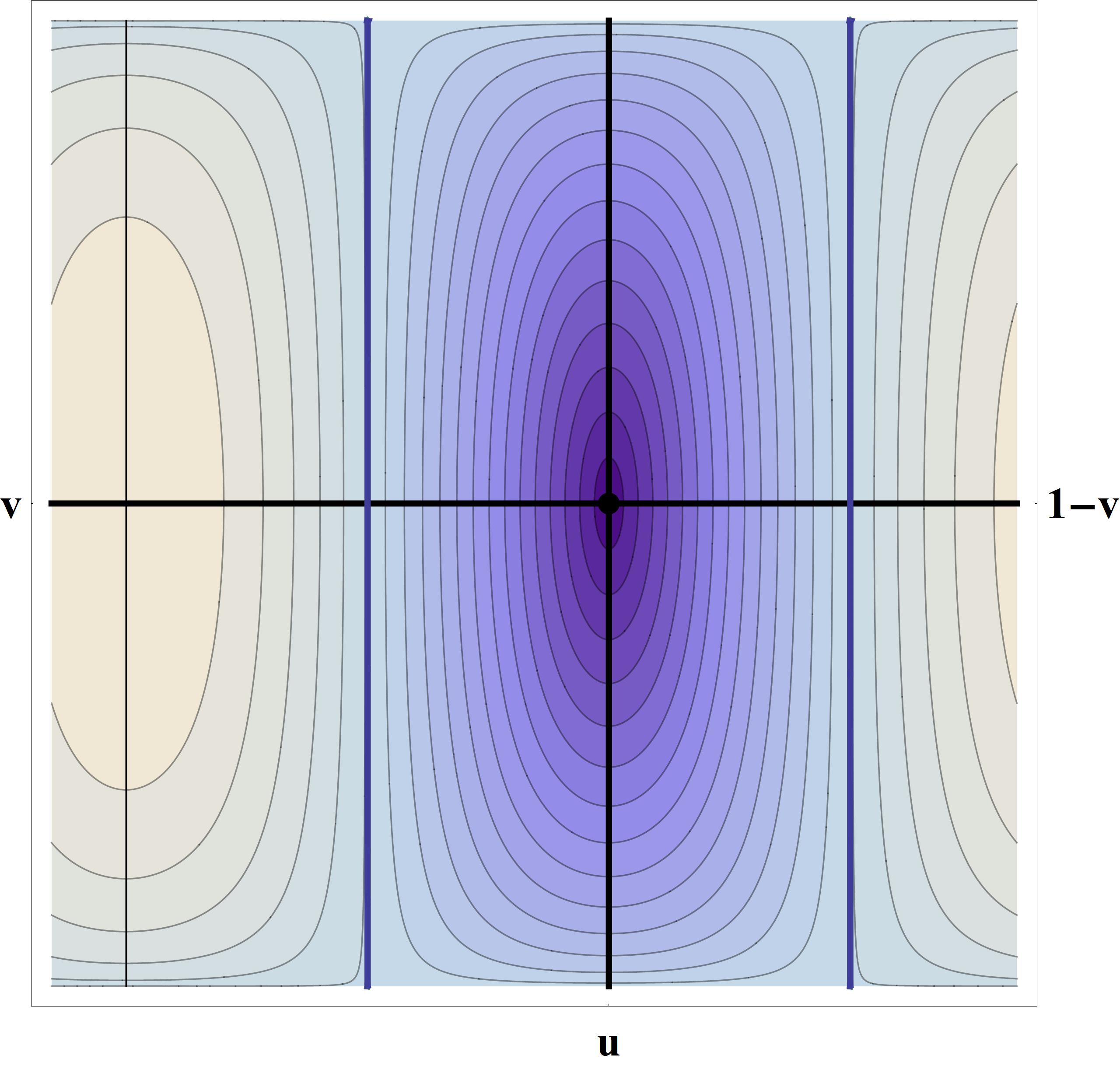}
\caption{\label{fig:curvature} Contour plot of the distance function $\| \bsw - \bsPhi(\alpha,\tau)\|$ for various $\bsw$. Level curves for distances $\sqrt{2}$ and $2\sqrt{v}$, $2\sqrt{1-v}$ are emphasised. The solid curve shows vanishing curvature.}
\end{center}
\end{figure}

For $v = 1/2$ the polynomial $Q$ reduces to
\begin{equation*}
Q( x ) = x^3 + p x, \qquad \text{where} \quad p = - \frac{1+2 H^2}{1+H^2} = - \frac{1 + 2 \left( 1 - 2 \tau \right)^2}{1 + \left( 1 - 2 \tau \right)^2}.
\end{equation*}
The solutions $\pm1$ (if $\tau = 1/2$) correspond to $\bsPhi(\alpha,\tau) = \pm \bsw$ and can be discarded, since we assumed that the spherical cap is neither a point nor the whole sphere. The solution zero yields that $\cos(2\pi(u-\alpha)) = 0$, which in turn shows that the zeros of the curvature \eqref{eq:curvature} form the vertical lines at $\alpha = u \pm 1/4 \mod 1$ if $v = 1 / 2$. 

Let $v \neq 1/2$. Suppose $Q$ has a zero at $\pm1$. Then 
\begin{equation*}
0 = Q(1) = 1 + p + q = \pm \frac{\left( 1 - 2 v + B H \right)^2}{B^2 \left( 1 + H^2 \right)} = \pm \frac{\left( 1 - 2 v + \frac{\sqrt{\left( 1 - v \right) v}}{\sqrt{\left( 1 - \tau \right) \tau}} \left( 1 - 2 \tau \right) \right)^2}{B^2 \left( 1 + H^2 \right)}
\end{equation*}
which can only happen when $\tau = v$. This implies that $\bsPhi(\alpha,\tau) = \bsw$, which is excluded by our assumptions. Suppose $Q$ has a zero at $0$. Since $v \neq 1/2$, this can only happen when $\tau = 1/2$.

Having established that $-1$, $0$ (except when $\tau = 1/2$), and $1$ cannot be zeros of the polynomial $Q$, we use Storm's theorem to show that the polynomial $Q$ has precisely one solution either in the interval $(-1,0)$ or in the interval $(0,1)$ if $\tau \neq 1/2$, cf. Table~\ref{tbl:sturm}. First, we generate the canonical Storm chain by applying Euclid's algorithm to $Q$ and its derivative:
\begin{align*}
p_0(x) &= Q(x) = x^3 + p x + q, \\
p_1(x) &= Q^\prime(x) = 3 x^2 + p, \\
p_2(x) &= p_1(x) q_0(x) - p_0(x) = - \frac{2p}{3} x - q, \\
p_3(x) &= p_2(x) q_1(x) - p_1(x) = - p - \frac{27 q^2}{4 p^2} = \frac{- 4 p^3 - 27 q^2}{4 p^2} = \frac{\discr(Q)}{4p^2} > 0, \\
p_4(x) &= 0.
\end{align*}
Let $\sigma(x)$ denote the number of sign changes (not counting a zero) in the sequence 
\begin{equation*}
\{p_0(x), p_1(x), p_2(x), p_3(x)\}.
\end{equation*}
For $x = 0$ we obtain the canonical Storm chain $\{q, p, -q, \discr(Q) / (4 p^2)\}$ and we conclude that $\sigma(0) = 2$ for $(1/2 - v)(1/2 - \tau) > 0$ and $\sigma(0) = 1$ otherwise. For $x=1$ we have
\begin{align*}
p_0(1) &= 1 + p + q = - \frac{\left( 1 - 2 v + B H \right)^2}{B^2 \left( 1 + H^2 \right)} < 0, \\
p_1(1) &= 3 + p = \frac{2 + H^2}{1 + H^2} - \frac{\left( 1 - 2 v \right)^2}{B^2 \left( 1 + H^2 \right)} > 0 \qquad \text{iff $(1 - \tau ) \tau < 3 ( 1 - v ) v$,} \\
p_2(1) &= - \frac{2p}{3} - q = \frac{2 \left[ B^2 \left( 1 + 2 H^2 \right) + \left( 1 - 2 v \right)^2 - 3 B H \left( 1 - 2 v \right) \right]}{3 B^2 \left( 1 + H^2 \right)} \MARKED{Black}{> 0}, \\
p_3(1) &= \frac{\discr(Q)}{4p^2} > 0.
\end{align*}
\MARKED{Black}{(The positivity of $p_2(1)$ has been verified using Mathematica.)} Hence, in all three cases $p_1(1) < 0$, $p_1(1) = 0$, and $p_1(1) > 0$ one gets $\sigma(1) = 1$. For $x = -1 $ we have
\begin{align*}
p_0(-1) &= -1 - p + q = \frac{\left( 1 - 2 v + B H \right)^2}{B^2 \left( 1 + H^2 \right)} > 0, \\
p_1(-1) &= 3 + p = \frac{2 + H^2}{1 + H^2} - \frac{\left( 1 - 2 v \right)^2}{B^2 \left( 1 + H^2 \right)} > 0 \qquad \text{iff $(1 - \tau ) \tau < 3 ( 1 - v ) v$,} \\
p_2(-1) &= \frac{2p}{3} - q = - \frac{2 \left[ B^2 \left( 1 + 2 H^2 \right) + \left( 1 - 2 v \right)^2 + 3 B H \left( 1 - 2 v \right) \right]}{3 B^2 \left( 1 + H^2 \right)} < 0, \\
p_3(-1) &= \frac{\discr(Q)}{4p^2} > 0.
\end{align*}
Here, we obtain $\sigma(-1) = 2$. Thus, by Storms Theorem, the difference $\sigma(-1) - \sigma(0)$ gives the number of real zeros of $Q$ in the interval $(-1,0]$ and $\sigma(0) - \sigma(1)$ is the number of zeros in $(0,1]$, see Table~\ref{tbl:sturm}.

\begin{table}[h!t]
\caption{\label{tbl:sturm} The number of real zeros of $Q$ in the intervals $(-1,0)$ and $(0,1)$ as it follows from Sturm's Theorem.} 
\begin{tabular}{cc|ccc|cc}
\multicolumn{2}{c|}{Range of $v$ and $\tau$} & $\sigma(-1)$ & $\sigma(0)$ & $\sigma(1)$ & $(-1,0)$ & $(0,1)$ \\
\hline
$0 < v < 1/2$, & $0 < \tau < 1/2$ & $2$ & $2$ & $1$ & $0$ & $1$ \\
$0 < v < 1/2$, & $1/2 < \tau < 1$ & $2$ & $1$ & $1$ & $1$ & $0$ \\
$1/2 < v < 1$, & $0 < v < 1/2$    & $2$ & $1$ & $1$ & $1$ & $0$ \\
$1/2 < v < 1$, & $1/2 < \tau < 1$ & $2$ & $2$ & $1$ & $0$ & $1$
\end{tabular}
\end{table}

Having established that the polynomial $Q$ has to every $0 < \tau < 1$ precisely one zero in the interval $(-1,1)$ (cf. Table~\ref{tbl:sturm} and previous considerations), it follows that to each such zero $x = x(\tau)$ there correspond two values of $\alpha$ by means of the trigonometric equation
\begin{equation} \label{eq:cos.rel}
\cos( 2 \pi ( u - \alpha ) ) = x( \tau ).
\end{equation}
Because of the continuity of the coefficients in the polynomial $Q$ (if $0 < \tau < 1$) the zero $x(\tau)$ is also changing continuously and so are the solutions $\alpha_1$ and $\alpha_2$. A jump can happen when they are taken modulo $1$. We further record that along the vertical lines $\alpha = u \pm 1/4 \mod 1$ one has
\begin{equation*}
\kappa( u \pm 1/4, \tau ) = - 16 \pi^2 \frac{\left( 1 - v \right) v \left( 1 - 2 v \right) \left( 1 - 2 \tau \right)}{\left[ 1 - 4 \left( 1 - v \right) v \left( 1 - 4 \pi^2 \left( 1 - \tau \right) \tau \right) \right]^{3/2}}
\end{equation*}
which vanishes at $\tau = 1/2$ and along the vertical lines $\alpha = u \pm 1/2 \mod 1$ and $\alpha = u$ one has
\begin{equation*}
\kappa( u, \tau ) = \kappa( u \pm 1/2, \tau ) = - 8 \pi^2 \frac{\sqrt{\left( 1 - v \right) v} \left( 1 - \tau \right) \tau}{\left| \left( 1 - 2 v \right) \sqrt{\left( 1 - \tau \right) \tau} + \left( 1 - 2 \tau \right) \sqrt{\left( 1 - v \right) v} \right|}
\end{equation*}
which vanishes only as $\tau \to 0$ or $\tau \to 1$ (if $0 < v < 1$). When identifying the left and right side of the unit square, these two lines separate the two solutions of \eqref{eq:cos.rel} in such a way that in each part the points at which $\kappa(\alpha, \tau)$ vanishes form a connected curve varying about the ``base lines'' $\alpha = u \pm 1/4 \mod 1$. It follows that these curves (together with the boundary of the cylinder) divide the cylinder into two parts in each of which the curvature $\kappa(\alpha, \tau)$ has the same sign. The shapes of these curves do not change when $\bsw$ is rotated about the polar axis. We may fix $u = 1/2$ and because of the symmetries (including relation $Q(\tau; x ) = - Q(1-\tau; -x )$) it suffices to consider the curve of the zeros of $\kappa(\alpha, \tau)$ for $0 < v < 1/2$ (recall that these curves are vertical lines for $v = 1/2$) and $0 < \tau < 1/2$ which lies in the strip $0 < \alpha < 1/2$. We know that the zero $x = x(\tau)$ of $Q$ (we are interested in) in the given setting is in $(0,1)$ (cf. Table~\ref{tbl:sturm}). Using $\dot{x}$ to denote the derivative of $x$ with respect to $\tau$, implicit differentiation gives
\begin{equation} \label{eq:Q.prime.relation}
Q^\prime( x(\tau) ) \, \dot{x}(\tau) = - \dot{p}(\tau) \, x(\tau) - \dot{q}(\tau),
\end{equation}
where it can be easily seen that $Q^\prime( x(\tau) ) < 0$, since $x(\tau)$ is simple and $Q(x)$ has a negative global minimum for positive $x$. For $\tau$'s in $(0,1/2)$ at which $\dot{x}(\tau)$ vanishes one has 
\begin{equation} \label{eq:x.crit}
x(\tau) = - \dot{q}(\tau) / \dot{p}(\tau) = \frac{1-6 \left( 1 - \tau \right) \tau}{1 - 6 \left( 1 - v \right) v} \frac{1 - 2 v}{1 - 2 \tau} \sqrt{\frac{\left( 1 - v \right) v}{\left( 1 - \tau \right) \tau}},
\end{equation}
which follows by substituting 
\begin{align*}
\dot{p}( \tau ) = - \frac{\left( 1 - 6 \left( 1 - v \right) v \right) \left( 1 - 2 \tau \right)}{2 \left( 1 - v \right) v \left( 1 - 2 \left( 1 - \tau \right) \tau \right)^2}, \qquad \dot{q}( \tau ) = \frac{\left( 1 - 6 \left( 1 - \tau \right) \tau \right) \left( 1 - 2 v \right)}{2 \left( 1 - \tau \right) \tau \left( 1 - 2 \left( 1 - \tau \right) \tau \right)^2} \sqrt{\frac{\left( 1 - \tau \right) \tau}{\left( 1 - v \right) v}}.
\end{align*}
For the second derivative of $x$ at such $\tau$'s we get 
\begin{equation*}
Q^\prime( x( \tau ) ) \, \ddot{x}(\tau) = - \ddot{p}( \tau ) \, x( \tau ) - \ddot{q}( \tau ) = \frac{1-2v}{4 \left( \left( 1 -  \tau \right) \tau \right)^{3/2} \sqrt{\left( 1 - v \right) v} \left[ 1 - 2 \tau \left( 2 - \tau \left( 3 - 2 \tau \right) \right) \right]}.
\end{equation*}
The square-bracketed expression is strictly monotonically decreasing on $(0,1/2)$ and evaluates to zero at $\tau=1/2$. Thus, the left-hand side has to be positive for all critical $\tau$ in $(0,1/2)$ which in turn implies that $\ddot{x}(\tau) < 0$ at such $\tau$'s. We conclude that $x(\tau)$ has a single maximum in $(0,1/2)$, since it cannot be constant as can be seen from \eqref{eq:x.crit} by evaluating $x(\tau)$ at that $\tau^\prime$ at which $1-6 \left( 1 - \tau^\prime \right) \tau^\prime = 0$ and $x(\tau) \to \infty$ as $\tau \to 1/2$. By \eqref{eq:cos.rel} (recall $u = 1/2$)
\begin{equation*}
- \cos( 2 \pi \alpha(\tau) ) = x(\tau)
\end{equation*}
and it follows that $\alpha(\tau)$ has a single maximum in $(0,1/2)$. 
\begin{prop}
The set of zeros of $\kappa(\alpha, \tau)$ and the horizontal sides of the unit square divide the unit square either into three parts of equal sign of $\kappa(\alpha, \tau)$ or in four parts, cf. Figure~\ref{fig:curvature}. 
\end{prop}

Using the trigonometric method, we obtain an explicit expression of the zero of $Q$ in $(-1,1)$. The change of variable $x = 2 \sqrt{-p / 3} \, \cos \theta$ gives the equivalence
\begin{equation*}
Q( x ) = 0 \qquad \text{if and only if} \qquad \cos( 3 \theta ) = \frac{3 q}{2 p} \sqrt{\frac{3}{-p}}
\end{equation*}
and therefore
\begin{equation*}
x( \tau ) = 2 \sqrt{- p( \tau ) / 3} \, \cos\Big( \frac{1}{3} \arccos\Big( \frac{3q( \tau )}{2p( \tau )} \sqrt{\frac{3}{-p( \tau )}} \Big) - \frac{2\pi}{3} \Big).
\end{equation*}
(The discarded solutions are either smaller or larger than the given one. By our reasoning, they have to lie outside the interval $(-1,1)$.) \MARKED{Black}{A limit process shows that $x(\tau) \to 0$ as $\tau \to 0$ or $\tau \to 1$. Hence, when moving towards the upper or lower side of the square along the curve of zeros of the curvature, one approaches the corresponding ``base line'' $\alpha = u \pm 1/4 \mod 1$.}

Eliminating the trigonometric term, along a level curve with parameter $t$ (cf. \eqref{eq:implicit.representation}) we have
\begin{equation*}
\begin{split}
F_{\alpha\alpha} F_{\tau}^2 &- 2 F_{\alpha \tau} F_{\alpha} F_{\tau} + F_{\tau\tau} F_{\alpha}^2 = - 16 \pi^2 \Bigg[ \left( \frac{1}{\left( 1 - \tau \right)^2} + \frac{1}{\tau^2} \right) X^3 \\
&\phantom{=}- 8 \left( 1 + 3 \frac{\left( 1 - v \right) v}{\left( 1 - \tau \right) \tau} \left( 1 - 2 \tau \right)^2  \right) X + 64 \left( 1 - v \right) v \left( 1 - 2 v \right) \left( 1 - 2 \tau \right) \Bigg],
\end{split}
\end{equation*}
where $X = t - ( 1 - 2 v ) ( 1 - 2 \tau )$. Reordering the terms and using the substitution $G = t - ( 1 - 2 v )$, we arrive at
\begin{equation}
\begin{split} \label{eq:kapp.along.curve}
F_{\alpha\alpha} &F_{\tau}^2 - 2 F_{\alpha \tau} F_{\alpha} F_{\tau} + F_{\tau\tau} F_{\alpha}^2 = - \frac{16}{\left( 1 - \tau \right)^2 \tau^2} \Big[ G^3 - 2 G \left( 3 + G^2 + 3 G t - 3 t^2 \right) \tau \\
&\phantom{=}+ 2 \left( 2 t + 6 G^2 t - 2 t^3 + 3 G \left( 3 - t^2 \right) \right) \tau^2 - 4 \left( 5 t - 3 t^3 + 3 G \left( 1 + t^2 \right) \right) \tau^3 + 16 t \tau^4 \Big].
\end{split}
\end{equation}
The zeros of the numerator of the curvature \eqref{eq:curvature} are determined by a polynomial in $\tau$ of degree $4$. Thus, there can be at most four pairs (symmetry with respect to $\alpha = u$) of points on the level set at which the curvature vanishes.  For the sake of completeness, in a similar way one obtains
\begin{equation*}
\left( F_{\alpha}^2 + F_{\tau}^2  \right)^{3/2} = \frac{1}{\left( 1 - \tau \right)^3 \tau^3} \left[ \left( G - 2 t \tau \right)^2 - 16 \pi^2 \left( 1 - \tau \right)^2 \tau^2 \left( G^2 - 4 G t \tau - 4 \tau \left( 1 - t^2 - \tau \right) \right) \right]^{3/2}.
\end{equation*}

The critical curves when the boundary of the spherical cap passes through a pole are of particular interest. In the case of the North Pole (that is $t = 1 - 2 v$, or equivalently $G = 0$), the curvature along the corresponding level curve reduces to
\begin{equation} \label{eq:curvature.special}
\kappa( \tau ) = - 16 \pi^2 \frac{\left( 1 - 2 v \right) \left( 1 - \tau \right) \left[ 2 \left( 1 - v \right) v - \left( 1 + 6 \left( 1 - v \right) v \right) \tau + 2 \tau^2 \right]}{\left[ \left( 1 - 2 v \right)^2 - 16 \pi^2 \left( 1 - \tau \right)^2 \tau \left( \tau - 4 \left( 1 - v \right) v \right) \right]^{3/2}}.
\end{equation}
For given $0 < \tau < 1$, the corresponding value(s) of $\alpha$ can be obtained from the relation
\begin{equation} \label{eq:curvature.special.cos.rel}
\cos( 2 \pi ( u - \alpha ) ) = \frac{\left( 1 - 2 v \right) \tau}{2 \sqrt{\left( 1 - v \right) v \left( 1 - \tau \right) \tau}}.
\end{equation}
From \eqref{eq:implicit.representation} it follows that the range for $\tau$ is $(0,\tau_1]$ with $\tau_1=4v ( 1 - v)$. For future reference we record that for $0 < v < 1$ with $v \neq 1/2$ the curvature \eqref{eq:curvature.special} vanishes only for 
\begin{equation} \label{eq:tau.critical}
\tau_v = \frac{1}{4} \left( 1 + 6 \left( 1 - v \right) v - \sqrt{1 - 4 \left( 1 - v \right) v \left( 1 - 9 \left( 1 - v \right) v \right)} \right).
\end{equation}
\MARKED{Black}{(The other solution lies outside the interval $[-1,1]$ as one can verify with Mathematica.)}

\begin{prop}
The curves of zeros of the curvature \eqref{eq:curvature} (as functions of $\tau$) assume their extrema at $\tau_v$ and $1-\tau_v$ with $\tau_v$ given in \eqref{eq:tau.critical}.
\end{prop}
\begin{proof}
Suppose that $u = 1/2$ and $0 < v < 1/2$. Then $0 < \tau_v < 1/2$. On observing that the right-hand side of \eqref{eq:curvature.special.cos.rel} for $\tau = \tau_v$ is also the zero $x(\tau_v)$ in the interval $(0,1)$ of the polynomial $Q$, it can be verified \MARKED{Black}{with the help of Mathematica} that the right-hand side of \eqref{eq:Q.prime.relation} vanishes and therefore $\dot{x}( \tau_v ) = 0$; that is, the zero $x(\tau)$ is extremal at $\tau = \tau_v$. Using the symmetry relation $Q( \tau; x ) = - Q( 1 - \tau; - x )$, the zero $x(\tau)$ is also extremal at $\tau = 1 - \tau_v$. By means of \eqref{eq:curvature.special.cos.rel} this translates into extrema of the curve of zeros of the curvature \eqref{eq:curvature}. A shift in $u$ (rotation of $\bsw$ about the polar axis) does not change the shape of the level curves and the general result follows.
\end{proof}

Substituting \eqref{eq:tau.critical} into the right-hand side of \eqref{eq:curvature.special.cos.rel} gives the extremal value a zero of $Q$ in $(-1,1)$ can assume, also cf. \eqref{eq:x.crit}:
\begin{equation*}
x(\tau_v) = \frac{1 - 2 v}{\sqrt{1 + 2 \left( 1 - v \right) v + \sqrt{1 - \left( 1 - v \right) v \left( 9 \left( 1 - 2 v \right)^2 - 5 \right)}}}.
\end{equation*}
It can be shown that $x(\tau_v)$ is a strictly monotonically decreasing function in $v$ which is symmetric with respect to $v = 1/2$. Hence $| x(\tau_v) | \leq x(0^+) = 1 / \sqrt{2}$. Using this bound in \eqref{eq:curvature.special.cos.rel} yields that $| \alpha - ( u \pm 1 / 4 ) | \leq 1/8$ (when wrapping around).

When moving along the critical level curve towards the lower side of the unit square, which is associated with the North Pole, we have
\begin{equation*}
\lim_{\tau \to 0} \kappa( \tau ) = - 32 \pi^2 \frac{1-2v}{| 1 - 2 v |^3} \left( 1 - v \right) v, \qquad 0 < v < 1, v \neq 1/2.
\end{equation*}
We conclude that the critical level curve with $t = 1 - 2v$ (associated with the North Pole) has precisely one symmetric (in the cylindrical view) pair of intersection points with the curves of zeros of the curvature function \eqref{eq:curvature} at $\tau = \tau_v$ in the strip $0 < \tau < 1$. A similar result holds for the level curve associated with the South Pole ($t = -( 1 - 2 v )$). 

\begin{prop} \label{prop:intersections}
Let $0 < v < 1/2$. Then the level curves with $t$ in the range $-( 1 - 2 v ) \leq t \leq 1 - 2 v$ have precisely one symmetric (in the cylindrical view) pair of intersection points with the curve of zeros of the curvature function \eqref{eq:curvature}. For $t$ in the ranges $-1 < t < - ( 1 - 2 v)$ or $1 - 2 v < t < 1$ there are either no intersection points, one pair of \MARKED{Black}{tangential} points, or two pairs.
\end{prop}

The analogue result holds for $1/2 < t < 1$. (For $v = 1/2$ the level curves for distance $\sqrt{2}$ are the verticals at $\alpha = u \pm 1/4 \mod 1$ and coincide with the curve of zeros of $\kappa(\alpha, \tau)$ and also coincide with the critical curves. The other level curves have no intersections.)

\begin{proof}
Without loss of generality assume that $u = 1/2$. We have already established that either critical level curve has precisely one pair of symmetric (with respect to $\alpha = u$) intersection point with the two curves of zero curvature about the base lines $\alpha = u \pm 1/4$ at the values $\tau = \tau_v$ and $\tau = 1 - \tau_v$. These parameter values also give the position of the extrema of the zero curves, cf. Figure~\ref{fig:curvature}. The left zero curve $\mathcal{Z}$ is increasing for $\tau$ in $(0,\tau_v)$, decreasing for $\tau$ in $(\tau_v, 1 - \tau_v)$ and increasing again for $\tau$ in $(1 - \tau_v, 1)$. 

Let $-(1-2v) < t < 1 - 2v$ and $\Gamma_t$ denote the left half of the corresponding level curve starting at the left side at some point $(0,\tau_1)$ and ending at some point $(u,\tau_2)$. (The other half is symmetric.) We note that the part where the zero curve is increasing is contained in the regions separated off by the critical level curves. Thus, an intersection between $\mathcal{Z}$ and $\Gamma_t$ can only occur for $\tau$ in the interval $[t_v, 1 - t_v]$. The curvature along $\Gamma_t$ changes continuously (cf. \eqref{eq:kapp.along.curve} and subsequent formula) from negative to positive value. Hence, there is an intersection point of $\mathcal{Z}$ and $\Gamma_t$ \MARKED{Black}{and the $\Gamma_t$ cannot change abruptly}. In particular, both $\mathcal{Z}$ and $\Gamma_0$\footnote{The spherical cap is the half-sphere.} pass through $(1/4,1/2)$, which is their only intersection point because for $\tau \neq 1/2$ the vertical line $\alpha = 1/4$ separates both curves. A $\Gamma_t$ with $0 < t < 1 - 2 v$ ($-(1 - 2 v) < t < 0$) has to intersect $\mathcal{Z}$ in the strip $1/4 < \alpha < 1/2$ ($0 < \alpha < 1/4$). Inspecting the partial derivatives of $F$ (cf. \eqref{eq:partial.derivatives}) it follows that the Gradient of $F$ at the intersection point, which is the outward normal at the level curve $\Gamma_t$, points into the upper left part; that is the tangent vector at $\Gamma_t$ at the intersection point shows to the right whereas the tangent vector at $\mathcal{Z}$ at this point shows to the left. Moreover, if $0 < t < 1 - 2 v$, then the curve $\Gamma_0$ separates $\mathcal{Z}$ and $\Gamma_t$ for $\tau \geq 1/2$ and in the remaining part both curves $\mathcal{Z}$ and $\Gamma_t$ bend away from each other because $\mathcal{Z}$ is decreasing with growing $\tau$ and the curvature along $\Gamma_t$ becomes positive. Consequently, there is only one intersection point of $\Gamma_t$ and $\mathcal{Z}$. A similar argument holds for $- ( 1 - 2 v) < t < 0$. 

Let $1 - 2 v < t < 1$. Let $\Gamma_t$ denote the left half of the level curve. For $t$ sufficiently close to $1$ there is no intersection with $\mathcal{Z}$ and the level curve is convex. If $\Gamma_t$ and $\mathcal{Z}$ intersect in only one point, then the level curve is still convex, since the curvature function $\kappa(\alpha, \tau)$ has positive sign in the section between $\mathcal{Z}$ and the vertical line $\alpha = u$. In this case $\mathcal{Z}$ and $\Gamma_t$ share a common tangent at the intersection point. If the curvature along $\Gamma_t$ changes its sign to negative, then it has to become positive again, since it is positive when crossing the vertical $\alpha = u$. But after changing back to positive curvature, both curves are bending away from each other. So, there can be no other intersection point.

By symmetry with respect to the line $\alpha = u$ one has pairs of symmetric intersection points.

Shifting $u$ does not change the form of the curves and their relative positions. This completes the proof.
\end{proof}

\section{Proofs}
\label{sec:proofs}

\begin{proof}[Proof of Lemma~\ref{lem2}]
For $t=1$ the spherical cap is a point and for $t=-1$ it is the whole sphere. Their pre-images (a point and the whole unit square) are convex. So, we may assume that $-1 < t < 1$.

{\bf Case (i):} Let $\bsw$ be either the North or the South Pole. Then the pre-images of the boundary of spherical caps centred at $\bsw$ are horizontal lines in the unit square. Hence, the pre-image of such a spherical cap is convex.

{\bf Case (ii):} Let $\bsw$ be on the equator (that is $v = 1/2$). We know that the curvature \eqref{eq:curvature} vanishes along the lines $\alpha = u \pm 1/4 \mod 1$. First, suppose that $u = 1/4$. Then the pre-image of any spherical cap centred at $\bsw$ with boundary points at most Euclidean distance $\sqrt{2}$ away from $\bsw$ is convex. For a larger spherical cap $C$ it follows that its complement $\overline{C}$ with respect to the sphere (centred at the antipodal point $-\bsw$) has the property that points on the boundary have distance $\leq \sqrt{2}$ from $-\bsw$. Hence, the pre-image of $\overline{C}$ is convex. When rotating $\bsw$ about the polar axis (that is shifting $u$), the vertical boundaries of the square cut these convex sets into two parts. We conclude that the pre-image of a spherical cap centred at $\bsw$ or its complement with respect to the sphere is the union of at most two convex sets.

{\bf Case (iii):} Let $\bsw$ be neither the poles nor located at the equator. Without loss of generality we may assume that $\bsw$ is in the upper half of the sphere; that is $0 < v < 1 / 2$. (Otherwise we can use reflection with respect to the equator.) 
First, let us consider the canonical position $u = 1/2$. Let $- ( 1 - 2 v ) \leq t \leq 1 - 2 v$. Then, by Proposition~\ref{prop:intersections}, there are precisely two (symmetric) points along the level curve at which the curvature vanishes, say at $(\alpha_1,\tau_t)$ and $(\alpha_2, \tau_t)$. This yields a decomposition of the unit square into three vertical rectangles such that either the part above or below the level curve is convex. Let $1 - 2 v < t < 1$. By Proposition~\ref{prop:intersections} the level curve is already convex or there are two pair of symmetric points at which the curvature along the level curve vanishes and a sign change occurs. Hence, there are numbers $\tau_1 < \tau_2$ such that the level curve is convex for $\tau \leq \tau_1$ and convex for $\tau \geq t_2$. The remaining middle part can be covered by a convex isosceles trapezoid which in turn can be split by some vertical line contained in the level set associated with the level curve. Thus, one has again two convex polygons which are divided into a convex and non-convex part by the level curve. A similar argument holds for $- 1 < t < - ( 1 - 2 v )$. 

A shift of $u$ does not increase the number of vertical rectangles needed for $0 \leq t \leq 1 - 2 v$. (In fact, one may even reduce the number of elements of the partition.) In the case $1 - 2 v < t < 1$ one may need to use a covering of the pre-image of the spherical cap with up to $7$ pieces. A more precise analysis is listed in Table~\ref{tbl:parts}.
\end{proof}

\begin{table}[h!t]
\caption{\label{tbl:parts} Worst-case admissible convex covering with $p$ part and $q$ of which are convex. The vertical lines show canonical positions of the vertical borders of $[0,1]^2$.}
\begin{minipage}{0.25\linewidth}
\begin{tabular}{c|cc|c}
  & $p$ & $q$ & $2p-q$ \\
\hline
A & $4$ & $2$ & $6$ \\
B & $5$ & $3$ & $7$ \\
\rowcolor[gray]{.8}
C & $7$ & $3$ & $11$ \\
D & $6$ & $3$ & $9$ \\
E & $6$ & $4$ & $8$
\end{tabular}
\end{minipage}
\begin{minipage}{0.6\linewidth}
\includegraphics[scale=.09]{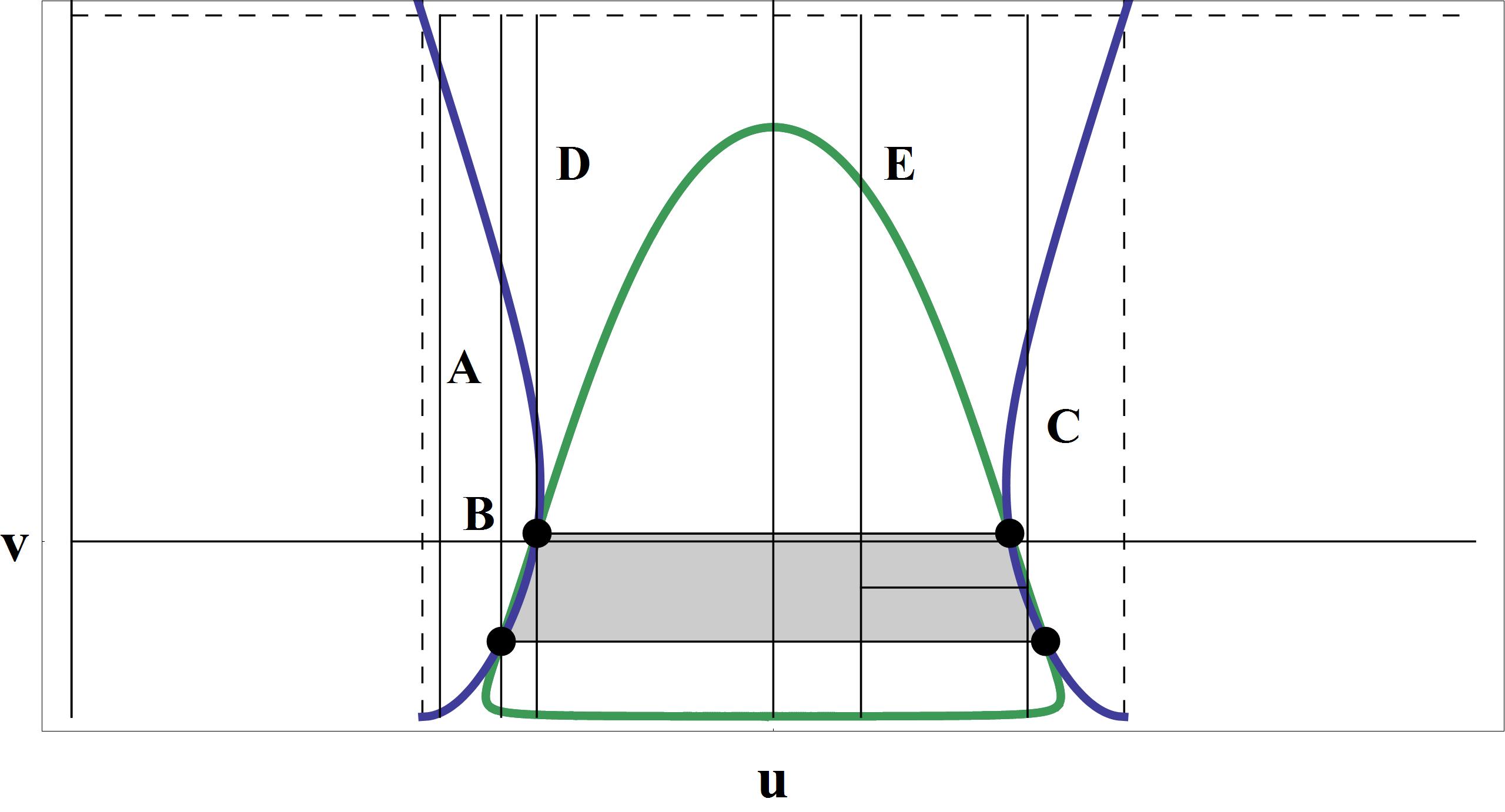}
\end{minipage}
\end{table}

\begin{proof}[Proof of Proposition~\ref{propo}]
Radon's theorem (see e.g. \cite[Theorem 4.1]{barv}) states that any set of $d+2$ points from $\mathbb{R}^d$ can be partitioned into two disjoint subsets whose convex hulls intersect. Particularly, let $A$ denote a set of 5 points on the sphere. Then by Radon's theorem there exists a partitioning of $A$ into disjoint subsets whose convex hulls intersect. Thus the set $A$ can not be shattered by the class of halfspaces. Since every spherical cap is the intersection of the sphere with an appropriate halfspace, the set $A$ can also not be shattered by the class of spherical caps. Thus the VC dimension of the class of spherical caps is at most 5.\\
On the other hand, let the set $\hat{A}$ consist of the points of a regular simplex, which lie on the sphere. Then some simple considerations show that the set $\hat{A}$ is shattered by the class of spherical caps. Thus the VC dimension of the class of spherical caps (and therefore of course also the VC dimension of the class $\mathcal{C}$ of spherical caps for which the centre $\mathbf{w}$ and the height $t$ are rational numbers, which was used in Section \ref{sec:rand}) equals 5.
\end{proof}

\begin{proof}[Proof of Theorem \ref{tha1}]
As mentioned directly after the statement of Theorem \ref{tha1}, the lower bound in the theorem follows directly from (\ref{exp}). To prove the upper bound we use Theorem \ref{thmtala}. By Proposition \ref{propo} the VC dimension of the class $\mathcal{C}$ in Section \ref{sec:rand} is 5. For simplicity we assume that the constant in Theorem \ref{thmtala} is an integer. Then for any $s \geq 2 K$
$$
\IP \left\{ D(Z_N) \geq \frac{s}{\sqrt{N}} \right\} \leq \frac{1}{s} \left(\frac{K s^2}{5} \right)^{5} e^{-2s^2},
$$
and consequently we have
\begin{eqnarray*}
\mathbb{E} (D(Z_N)) & \leq & \frac{2K}{\sqrt{N}} + \sum_{s = 2K}^\infty \left( \frac{(s+1)}{\sqrt{N}} \cdot \IP \left\{ D(Z_N) \geq \frac{s}{\sqrt{N}} \right\} \right) \\
& \leq & \frac{2K}{\sqrt{N}} + \sum_{s = 2K}^\infty \frac{s+1}{s \sqrt{N}} \left(\frac{K s^2}{5} \right)^{5} e^{-2s^2} \\
& \leq & \frac{\hat{K}}{\sqrt{N}}
\end{eqnarray*}
for some appropriate constant $\hat{K}$. This proves the theorem.
\end{proof}

\begin{proof}[Proof of Theorem \ref{tha2}]
Let $C^* \subseteq \mathcal{C}$ denote a hemisphere, i.e. a spherical cap whose normalised surface area measure is $\sigma(C^*)=1/2$. By the central limit theorem for any $t \geq 0$
$$
\IP \left\{ \left| \frac{1}{N} \sum_{n=0}^{N-1} 1_{C^*} (X_n) - \sigma(C^*) \right| \leq t N^{-1/2} \right\} \to \frac{\sqrt{2}}{\sqrt{\pi}} \int_{-t}^t e^{-2u^2}~du \quad \textrm{as $N \to \infty$},
$$ 
and consequently, for any given $\varepsilon>0$ and sufficiently small $C_3(\varepsilon)>0$,
$$
\IP \left\{ \left| \frac{1}{N} \sum_{n=0}^{N-1} 1_{C^*} (X_n) - \sigma(C^*) \right| \leq C_3 N^{-1/2} \right\} \leq \varepsilon/2
$$
for sufficiently large $N$. Since $D(Z_N) \geq \left| \frac{1}{N} \sum_{n=0}^{N-1} 1_{C^*} (X_n) - \sigma(C^*) \right|$, this implies
\begin{equation} \label{lower}
\IP \left\{ D(Z_N) \leq C_3 N^{-1/2} \right\} \leq \varepsilon/2
\end{equation}
for sufficiently large $N$.\\
On the other hand, by Theorem \ref{thmtala} for any given $\varepsilon>0$ and sufficiently large $C_4(\varepsilon)$
\begin{equation} \label{upper}
\IP \left\{ D(Z_N) \geq \frac{C_4}{\sqrt{N}} \right\} \leq \varepsilon/2
\end{equation}
for sufficiently large $N$. Combining (\ref{lower}) and (\ref{upper}) we obtain
$$
\IP \left\{\frac{C_3}{\sqrt{N}} \leq D(Z_N) \leq \frac{C_4}{\sqrt{N}} \right\} \geq 1 - \varepsilon
$$ 
for sufficiently large $N$, which proves Theorem \ref{tha3}.
\end{proof}

\section{Acknowledgements}

The authors are grateful to Peter Kritzer and Friedrich Pillichshammer for inspiring suggestions. In particular, F. Pillichshammer pointed out Theorem~\ref{thm_larcher}.

\bibliographystyle{abbrv}
\bibliography{bibliography}

\end{document}